\documentclass[12pt]{article}

\usepackage{amsmath}
\usepackage{amsthm}
\usepackage{amsfonts}
\usepackage{url}
\usepackage{xcolor}
\usepackage{hyperref} 
\hypersetup{
    colorlinks,
    linkcolor={red!50!black},
    citecolor={blue!50!black},
    urlcolor={blue!80!black}
}
\usepackage[lined,boxed,linesnumbered,algoruled]{algorithm2e}
\usepackage{cleveref}
\usepackage{todonotes}
\usepackage{authblk}
\usepackage[margin=1in]{geometry}
\usepackage{multirow}
\usepackage{float} 
\usepackage[all]{nowidow}

\newcommand{\NN}{\mathbb{N}}

\newcommand{\RR}{\mathbf{R}}
\newcommand{\CC}{\mathbb{C}}

\newcommand{\PP}{\mathbb{P}}
\newcommand{\mr}{\mathbb{R}}
\newcommand{\mz}{\mathbb{Z}}
\newcommand{\mc}{\mathbb{C}}

\newcommand{\simp}{\mathbf{simp}}
\DeclareMathOperator{\vect}{\mathbf{vect}}
\DeclareMathOperator{\diam}{diam}
\DeclareMathOperator{\rank}{rank}

\theoremstyle{definition}
\newtheorem{theorem}{Theorem}[section]
\newtheorem{corollary}[theorem]{Corollary}
\newtheorem{proposition}[theorem]{Proposition}
\newtheorem{definition}[theorem]{Definition}

\newtheorem{lemma}[theorem]{Lemma}
\newtheorem{example}[theorem]{Example}
\newtheorem*{example*}{Example}
\theoremstyle{remark}
\newtheorem{remark}[theorem]{Remark}

\theoremstyle{definition}
\newtheorem*{definition*}{Definition}
\theoremstyle{plain}
\newtheorem*{theorem*}{Theorem}
\theoremstyle{plain}
\newtheorem*{corollary*}{Corollary}
\theoremstyle{plain}
\newtheorem*{proposition*}{Proposition}
\theoremstyle{plain}
\newtheorem*{conj*}{Conjecture}

\crefname{appendix}{Appendix}{Appendices}

\DeclareMathOperator{\wfs}{wfs}
\DeclareMathOperator{\lfs}{lfs}
\DeclareMathOperator{\lnfs}{lnfs}

\newcommand{\medial}{\mathcal{M}}
\DeclareMathOperator{\crit}{crit}
\DeclareMathOperator{\reach}{reach}
\newcommand{\numvars}{n}

\title{Computing geometric feature sizes \\ for algebraic manifolds}
\author[1]{Sandra Di Rocco}
\author[2]{Parker B. Edwards}
\author[3]{David Eklund}
\author[4]{Oliver~G\"{a}fvert}
\author[2]{Jonathan D. Hauenstein}
\affil[1]{Department of Mathematics, KTH, 10044, Stockholm, Sweden}
\affil[2]{Dept. of Applied and Computational Mathematics and Statistics, University of Notre Dame, Notre Dame, IN, USA}
\affil[3]{RISE, Research Institutes of Sweden, Isafjordsgatan 22, 16440, Kista, Sweden}
\affil[4]{Mathematical Institute, University of Oxford, United Kingdom}
\date{}

\begin{document}
\maketitle
\begin{abstract}\noindent
    We introduce numerical algebraic geometry methods for
    computing lower bounds on the reach, local feature size, and the weak feature size of the real part of an equidimensional and smooth algebraic variety using the variety's defining polynomials as input. 
    For the weak feature size, we also show that non-quadratic complete intersections generically have finitely many geometric bottlenecks, and describe how to compute the weak feature size directly rather than a lower bound in this case. 
    In all other cases, we describe additional computations that can be used to determine feature size values rather than lower bounds. We also present homology inference experiments that combine persistent homology computations with implemented versions of our feature size algorithms, both with globally dense samples and samples that are adaptively dense with respect to the local feature~size.
\end{abstract}
\section{Introduction}

Exploring the geometry of a given data set has proven to be a powerful tool in data analysis.
For example, topological data analysis (TDA) aims to recover 
topological information of a data set such as connectedness or holes in its shape~\cite{carltop, ChazalSurvey,ghristtop} and has been successfully applied to problems in a wide range of fields~\cite{GiuntiDatabase,hensel2021survey,wasserman2018survey}.
If the data set lies on a manifold that is algebraic, namely it lies on a 
geometric shape defined by algebraic equations, a more direct approach using computational algebraic geometry can be applied. 
In this case, the data set can be viewed as a sampling of the algebraic manifold, as shown in \Cref{fig:sample},
where it is important to find guarantees that the topology of the sample, i.e., 
the topology of the Vietoris-Rips complex defined by the data set (see Definition~\ref{def:complexes}) correctly estimates the topology of the underlying algebraic manifold. 
\begin{figure}[h!]
\centering
\includegraphics[scale=0.20]{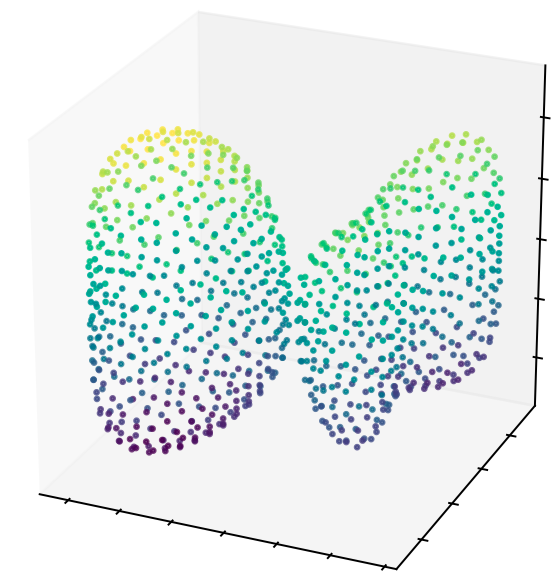}
\caption{Dense sample from a quartic surface.}\label{fig:sample}
\end{figure}

Topological and geometric data analysis algorithms frequently supply some form of the following guarantee: given a ``dense enough'' point sample from a space $X\subseteq\mr^\numvars$ as input, the algorithm correctly computes some geometric or topological property of $X$. The required density can be expressed in terms of certain invariants of the space $X$. The two most studied invariants are the \emph{reach}, introduced by Federer~\cite{FedererReach}, 
and the \emph{weak feature size}, introduced by Grove and Shiohama in the context of  Riemannian geometry  \cite{Grove93Critical,GroveShio77} and significantly expanded upon by Chazal and Lieutier for use in computational geometry~\cite{chazalwfs}. 
These invariants are of considerable importance  for persistent homology and reconstruction methods~\cite{AttaliReconstruction,BurgisserSemialg,chazalwfs,cl:reconstruction,cseh:stability,DeyReconstruction,KimHomReconstruction,Niyogi2008}.

In most settings, geometric feature sizes can only be estimated since a full specification of the space $X$ is not available. As a result, few examples of fully specified spaces with explicitly computed weak feature size have previously appeared. Algorithms computing these invariants and thus geometrical theories for efficient computations are an important area of study in applied geometry. This paper aims at providing some answers in this direction using  numerical algebraic geometric methods, e.g. see \cite{bertinibook,somnag}.

Throughout this paper, nonempty and compact algebraic manifolds \mbox{$X = V(F)\cap \mr^\numvars$}
are considered,  
where $F=\{f_1,f_2,\dots,f_m\}$ is a system consisting of polynomials
in $\mr[x_1,\dots,x_\numvars]$ and  $V(F) = \{x\in\mc^\numvars~|~F(x)=0\}.$
\Cref{sec:background} presents necessary background on feature sizes, persistent homology, and homology inference. 
The distance-to-X function \mbox{$d_X:\mr^\numvars\to\mr$}, 
defined as $d_X(z) = \inf_{x\in X}\|x-z\|,$
 is not differentiable everywhere in $\mr^\numvars$ for most spaces $X$.  Grove~\cite{Grove93Critical} constructs an analog of Morse theory defining \emph{critical points of $d_X$} or \emph{geometric bottlenecks} of $X$ as those points $z\in \mr^\numvars \setminus X$ which are in the convex hull of their closest points on $X$. The {\it weak feature size} is the infimum of all the \emph{critical values} of~$d_X$. The critical values of~$d_X$ are those values $d_X(z)$ where $z$ is a geometric bottleneck~(see~\Cref{def:wfs}). 

\begin{example*}
Consider the ellipsoid $X\subseteq\mr^3$ defined by $x_1^2 + x_2^2 + x_3^2/2=1$ as depicted in \Cref{fig:ellipsoid}. It has a single geometric bottleneck at the origin (red point). 
\begin{figure}[h]
\begin{center}
\includegraphics[scale=0.1]{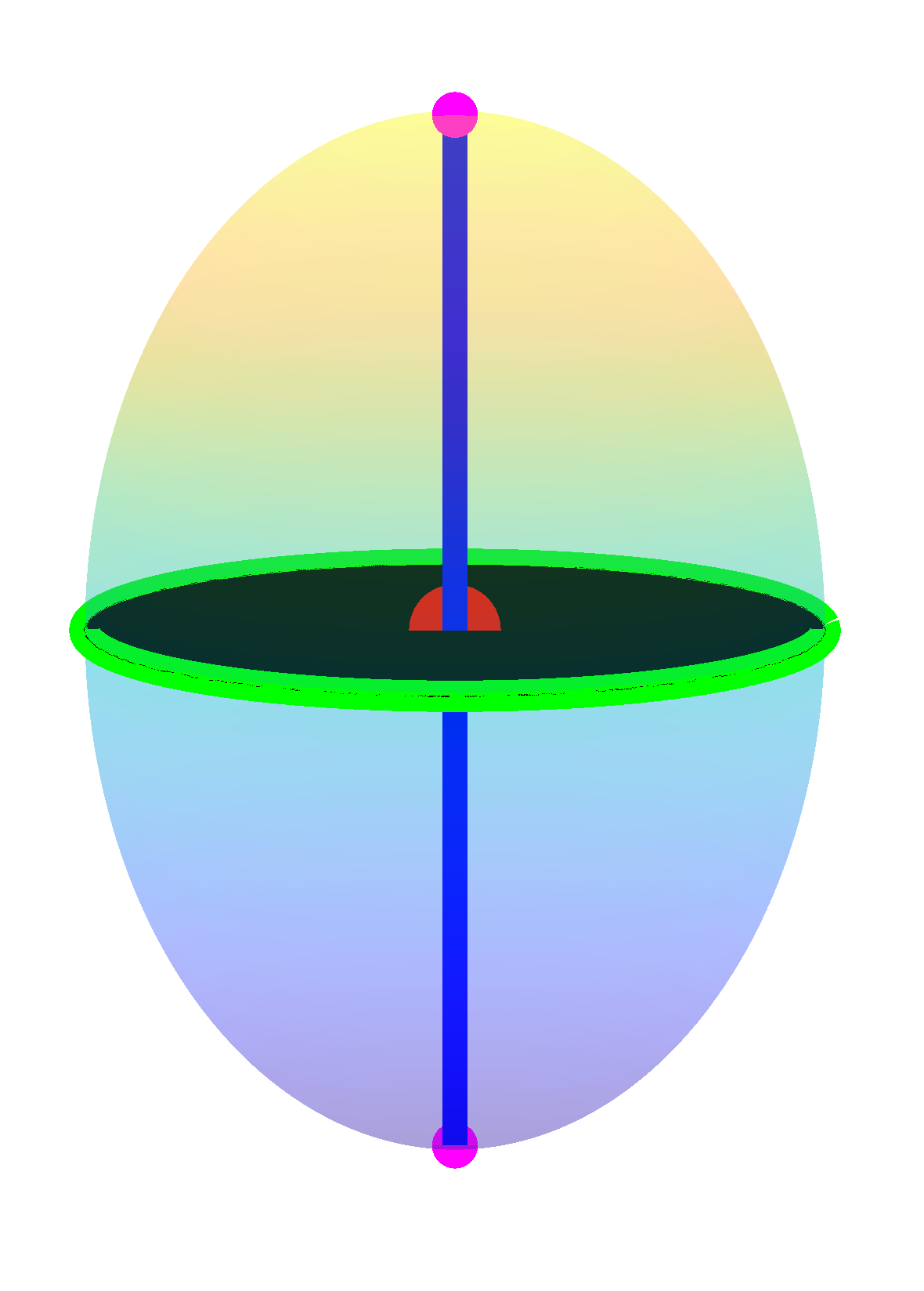}
\end{center}
\vspace{-1cm}
\caption{Ellipsoid \label{fig:ellipsoid}}
\end{figure} 

We will see in \Cref{sec:wfs} that the number of convex hulls of closest points which contain a geometric bottleneck crucially impacts computations. The ellipsoid in \Cref{fig:ellipsoid} is an example that poses difficulties, as it has a one-dimensional locus of convex hulls along the $(x_1,x_2)$-plane which contain the origin. They are depicted with black segments connecting green antipodal points on the unit circle in the $(x_1,x_2)$-plane.

Algebraic conditions also detect that the origin is contained in the convex hull of its furthest points on $X$, which lies along the $x_3$-axis with blue segments connecting the magenta points at $\left(0,0,\pm\sqrt{2}\right)$ in \Cref{fig:ellipsoid}.  
\end{example*}

Using a combination of geometric arguments, the Tarski-Seidenberg Theorem, and Sard's Theorem, Fu proved that the set of critical values of $d_X$ is finite when $X$ is semialgebraic~\cite{FuTubular1985}. This implies that the weak feature size is always positive. In the ellipsoid example, the weak feature size is $1$. This theorem strongly motivates studying the weak feature size as it applies even when 
$X$ is not smooth nor equidimensional. The proof, however, does not suggest a feasible algorithm for computing the critical values of $d_X$.

In \Cref{sec:lfs_compute}, 
we describe a method to compute the reach of $X$ as well as 
the \emph{local feature size}~\cite{amenta1999surface} of $X$ at a point $w\in\mr^\numvars$ given the defining polynomials~$F$ as~input. Numerical computations can compute these quantities to arbitrary precision via our approach.  Moreover, if the input depends on rational numbers, exactness 
recovery methods such as \cite{Exactness} can refine the numerical results to extract exact information.  For example, we use exactness recovery methods in \Cref{ex:cassini_def} to determine exact expressions for the reach of a particular space.

\begin{theorem*}[\ref{thm:reach_and_lfs}, \ref{cor:reach_and_lfs_hom}]
For both the reach and the local feature size, 
one can utilize the finite set of points computed via 
a single parameter homotopy~\cite{CoeffParam} on a polynomial system constructed
using first-order critical conditions to obtain a nontrivial lower 
bound.  Using additional reality testing, one can determine the value of
the reach and the local feature size.  
\end{theorem*}
To the best of our knowledge, these provide the first algorithms that can 
compute these quantities for algebraic manifolds of arbitrary codimension.

\Cref{sec:wfs} is dedicated to constructing a theory and algorithms for computing the weak feature size. We apply a wholly algebraic framework to this problem when $X$ is the real part of a smooth and equidimensional algebraic variety. The resulting theory yields an alternative proof of Fu's Theorem in this setting as well as a method for computing bounds on the weak feature size. 

\begin{theorem*}[\ref{thm:fin_many_crit_val}, \ref{cor:wfs_pd_cor}]
A lower bound on the weak feature size can be obtained using
the union of the finite set of points computed via
$\numvars$ parameter homotopies \cite{CoeffParam}.
Using additional reality testing, one can determine the value of the weak feature size.
\end{theorem*}

\begin{example*}
The ellipsoid example above has a geometric bottleneck with infinitely many closest points.
Although the previous theorem applies to that case, it is often more desirable 
from a numerical conditioning standpoint to consider nonsingular isolated solutions
to well-constrained systems.  Consider the perturbation defined by 
\mbox{$x^2 + y^2 + z^2/2 \begin{color}{blue} + xz/7\end{color} = 1$}, illustrated in Figure~\ref{fig:ellipsoid2}. In this case, only three convex hulls containing the geometric bottleneck at the origin contribute to algebraic computations. Black segments connect the origin to green points, which are distance minimizers. 
\begin{figure}[h]
\begin{center}
\includegraphics[scale=0.1]{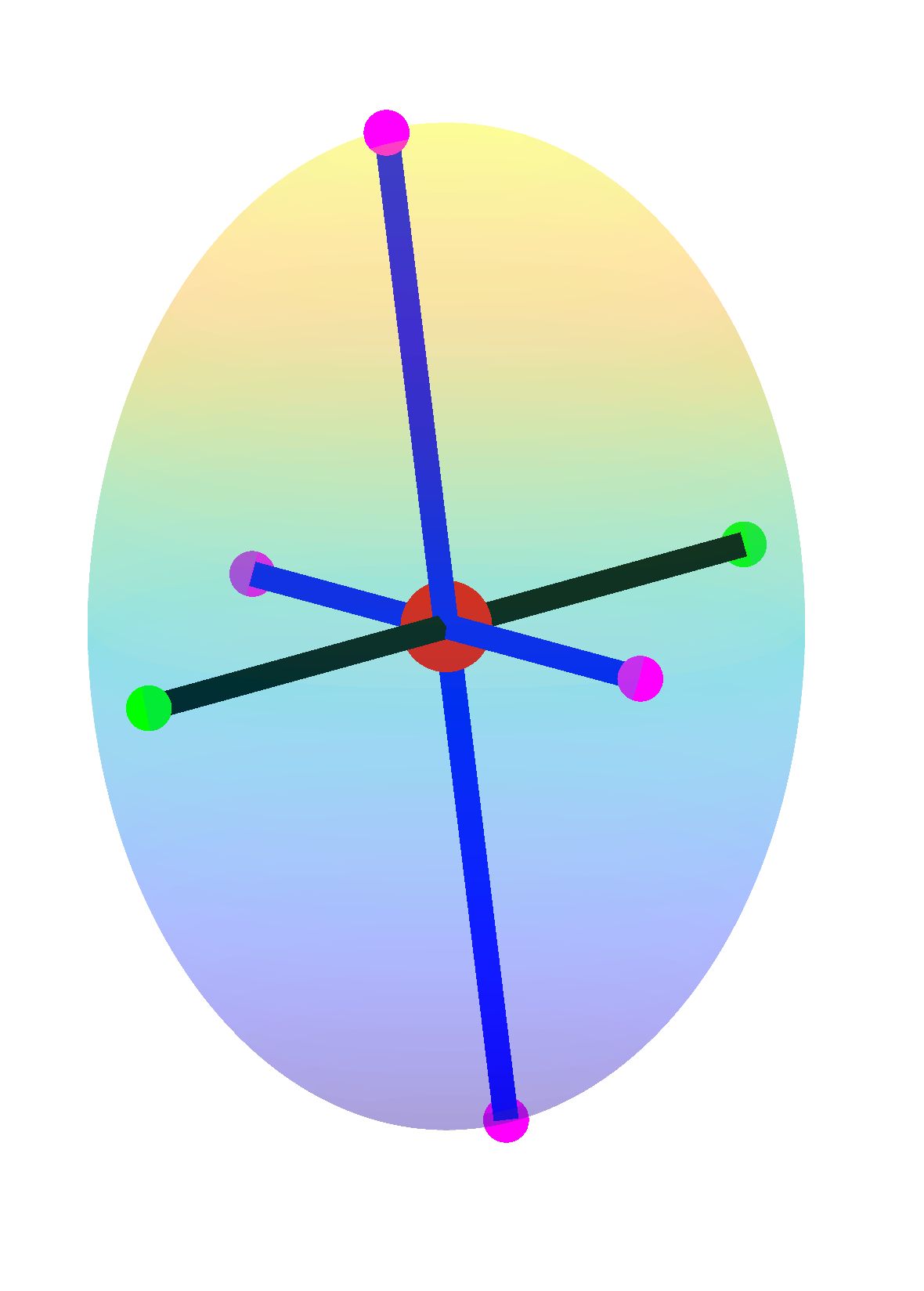}
\end{center}
\vspace{-1cm}
\caption{Perturbation of ellipsoid}\label{fig:ellipsoid2}
\end{figure}
\end{example*}

This example's behavior is the typical result of a perturbation in a rigorous sense. By applying the celebrated Alexander-Hirschowitz Theorem \cite{Alexander1995Interp}
on the expected dimension of the secant variety of the Veronese embedding, one obtains a description of the generic behavior of geometric bottlenecks
as summarized in the following.
\begin{theorem*}[\ref{thm:finitely_many_bottlenecks}]
Non-quadratic generic complete intersections  have finitely many critical points, i.e., finitely many geometric bottlenecks.
\end{theorem*}
As a consequence, we construct algorithms using homotopy continuation to compute the weak feature size with arbitrary precision. Examples are presented in \Cref{sec:examples}. A Julia package which implements these algorithms for general use via \texttt{HomotopyContinuation.jl}~\cite{Breiding2018HC} is available at \url{https://github.com/P-Edwards/HomologyInferenceWithWeakFeatureSize.jl}. 
We also use \texttt{Bertini}~\cite{bertini} implementations. Data, scripts, and input files for all examples are available at \url{https://github.com/P-Edwards/wfs-and-reach-examples}.

Our feature size algorithms comprise the final missing component of a homology inference pipeline that combines feature size computations, sampling methods for algebraic varieties~\cite{di2020sampling,dufresne2019sampling}, and persistent homology algorithms~\cite{ripser,chazprox,cseh:stability,Cufar2020ripserer}. In \Cref{sec:sparse}, we present homology inference results from an implemented version of this pipeline. Our feature size computations also provide a previously unavailable baseline to investigate the performance of methods which estimate feature sizes. As a proof of concept, we consider an ``adaptive'' subsampling method for persistent homology proposed by Dey et al.~\cite{DeyReconstruction} that estimates local feature sizes. Their approach has recently been expanded to a general framework for adaptive subsampling by Cavanna and Sheehy~\cite{cavanna2019,cavanna2020adaptive}. Using feature sizes computed via our new algorithms, we find evidence that the method of Dey et al. performs comparably to a baseline analog of Chazal and Lieutier \cite{cl:reconstruction} which requires directly computed local feature sizes.

\subsection{Related work} 

Recent work on computing feature sizes in the algebraic setting mostly focused on computing lower bounds for the reach, motivated by a result of Amari et al.~\cite[Thm.~3.4]{AmariEstimatingReach} which shows the reach of a compact manifold is determined by two distinct types of geometric behavior: regions of high curvature and ``bottleneck structures,'' which we call ``geometric 2-bottlenecks'' (\Cref{def:geom_bottleneck}). Breiding and Timme~\cite{reach-curve2022} observed that a straightforward computation can find the maximal curvature of an implicitly defined plane curve 
and Horobe\c{t}~\cite{horobet2021critical} studied the problem in greater generality by investigating an algebraic variety's critical curvature degree. Horobe\c{t} and Weinstein~\cite{Horobet2019offset} studied related theoretical problems in the context of ``offset filtrations'' and, in particular, showed that the reach is algebraic over $\mathbb{Q}$ for real algebraic manifolds defined by polynomials with rational coefficients. 
The third author~\cite{eklund2022numerical} studied computing 2-bottlenecks with numerical algebraic geometry while Weinstein together with the first and third authors~\cite{DiroccoBottleneckDegree}
developed formulas for the number of algebraic 2-bottlenecks of a smooth algebraic variety in terms of polar and Chern classes.
A subset of the present manuscript's authors~\cite{di2020sampling} 
show the special case of \Cref{thm:finitely_many_bottlenecks} for 2-bottlenecks
using a different approach.

Lowering the theoretical complexity of computing the Betti numbers and related invariants of semialgebraic sets from a list of defining polynomials comprises a rich and ongoing topic of study in real algebraic geometry, e.g., the references \cite{Basu2006Betti,Basu2021FirstHom,BurgisserSemialg} more extensively characterize recent progress in this area. 
The resulting algorithms are challenging to implement efficiently and,
to the best of our knowledge, no general implementations are available. 
We take a distinct approach to homology inference that is complimentary
by focusing on producing efficient implementations rather than lowering complexity bounds.

\section{Background and Preliminaries}
\label{sec:background}

The following summarizes the elements from the theory of distance functions and geometric feature sizes, particularly for semi-algebraic sets, necessary to state our results. We also recall how to combine feature size information with the ``persistent homology pipeline'' to compute the Betti numbers (with field coefficients) of a subspace of $\mr^\numvars$.

\subsection{Distance functions and geometric feature sizes} 
In this paper, the distance between two points $x,z\in\mr^\numvars$ uses the Euclidean distance: 
$$d(x,z)=\Vert x - z \Vert = \sqrt{\sum_{i=1}^\numvars (x_i-z_i)^2}.$$
For any nonempty subset $S\subseteq\mr^\numvars$, let $d_S:\mr^\numvars\to\mr$ denote the distance-to-$S$ function, namely $d_S(z) = \inf_{s\in S} \Vert s - z \Vert$. 
For any non-negative $\varepsilon$, let $S^\epsilon = d_S^{-1}[0,\epsilon]$ be the union of all closed $\epsilon$-balls centered at points of $S$. 
For a nonempty and compact subset $X\subseteq\mathbb{R}^\numvars$  and for any $z\in\mr^\numvars$, let $\pi_X(z)=\{x\in X~|~d_X(z)=d(x,z)\}$ be the set of the points in $X$ with minimal distance to $z$.
\begin{definition}\label{def:MedialAxis}
The \emph{medial axis} of $X$ is
\[ \medial_X=\overline{\{z\in\mr^\numvars~|~\#\pi_X(z)>1\}}\text{.} \]  Equivalently, $\medial_X$ is the (Euclidean) closure of the set of points in $\mr^\numvars$ that have at least $2$ closest points in $X$.
\end{definition}

Naturally, one can consider subsets of the medial axis based on the number of closest points.
That is, for $k\geq 2$, $\medial_{X,k} = \overline{\{z\in\mr^\numvars~|~\#\pi_X(z)\geq k\}}$
is the \emph{$k$-medial axis} where $\medial_X = \medial_{X,2}$, i.e., the medial axis is the $2$-medial axis.

\begin{definition}
The function $d_{\medial_X}:\mr^\numvars\to\mr\cup\{\infty\}$ is also called the \emph{local feature size} function of $X$ \cite{amenta1999surface}, denoted $\lfs$. 
For $w\in\mr^\numvars$, $\lfs(w)$ is the \emph{local feature size at $w$}. 
The function $\lfs$ takes value $\infty$ if $\medial_X=\emptyset$, e.g., if $X$ is convex. The \emph{reach}\label{reach} of $X$ \cite{FedererReach} is defined as \[\tau_X=\min_{x\in X} \lfs(x)\text{.}\] 
\end{definition}

\begin{definition}\label{def:wfs}
A point $z\in\mr^\numvars\setminus X$ is a \emph{critical point of $d_X$} \cite{Grove93Critical,GroveShio77} or a \emph{geometric bottleneck}\footnote{This term is new and agrees with that used in recent work on this subject in the algebraic context~\mbox{\cite{AmariEstimatingReach,reach-curve2022,di2020sampling,DiroccoBottleneckDegree,eklund2022numerical}}. It additionally distinguishes these points from other types of critical points which arise in this setting.} of $X$ if $z$ is in the convex hull of $\pi_X(z)$. 
The \emph{weak feature size} \cite{chazalwfs} of $X$ is defined as:
\[\wfs(X)=\inf_{z\in\crit(d_X)}~d_X(z) \]
where $\crit(d_X)$ denotes the set of critical points of $d_X$.
\end{definition}

Notice that $\crit(d_X)$ is a subset of $\medial_X$, so that $\tau_X\leq\wfs(X)$. Also notice that the above condition can be phrased in terms of well-centered simplices. The convex hull of a set of at most $\numvars+1$ affinely independent points in $\mr^\numvars$ is a \emph{well-centered simplex} if its circumcenter lies in its interior \cite{VanderZeeWellCentered2013}. A point $z\in\mr^\numvars \setminus X$ is a geometric bottleneck of $X$ if it is the circumcenter of a well-centered simplex with vertices in $\pi_X(z)$. 

\begin{example}\label{ex:cassini_def}
To illustrate the previous definitions, 
consider the plane curve $C\subseteq\mr^2$ defined by $d(x,p_1)^2\cdot d(x,p_2)^2 = 2$
where $p_1 = (1,0)$ and $p_2 = (-1,0)$.
The curve $C$ is called a \emph{Cassini oval} with 2 foci
and shown in Figure~\ref{fig:intro_cassinis}(a)
along with its medial axis (cyan curve)
and bottlenecks (red points).
These types of curves, with a concentration on examples bearing more resemblance to an ellipse than the one we consider, 
were proposed by Cassini in the late~$17^{\rm th}$~century as candidates for planetary orbits~\cite[p. 36]{Cassini1693}.\footnote{This 1693 publication of Cassini's is the earliest to which we could trace this example, but, e.g., Yates~\cite{Yates1947Curves} dates Cassini's study of these ovals to the earlier date of 1680 without citation.}
The medial axis $\medial_C$ consists of three segments along the coordinate axes, namely
$$\begin{array}{l}
(a,0) \hbox{~~~for~~~} a \in \left[-\sqrt{2\sqrt{2}-2},\sqrt{2\sqrt{2}-2}\right] \hbox{~~~and~~~} \\[0.1in]
(0,b) \hbox{~~~for~~~} b\in\left(-\infty,-\sqrt{ 2\sqrt{2}+2}\right]\cup\left[\sqrt{2\sqrt{2}+2},\infty\right).
\end{array}$$
The reach is $\tau_C = \sqrt{\sqrt{2}-1} \approx 0.6436$ attained at the origin and $\left(\pm \sqrt{2\sqrt{2}-2},0\right)$. 
There are three bottlenecks, namely
the origin and $(\pm \sqrt{1/2},0)$
each with two closest points in $C$,
with the weak feature size being \mbox{$\wfs(C) =  \sqrt{\sqrt{2}-1}$} attained at the origin.  
Hence, $\tau_C = \wfs(C)$.

Similarly, consider the plane curve $C'\subseteq\mr^2$
defined by $d(x,r_1)^2\cdot d(x,r_2)^2\cdot d(x,r_3)^2 = 2$
where $r_1 = (1,0)$, $r_2 = (-1/2,\sqrt{3}/2)$,
and $r_3 = (-1/2,-\sqrt{3}/2)$.
The curve $C'$ is called a \emph{Cassini oval} with 3 foci 
and is shown in Figure~\ref{fig:intro_cassinis}(b)
along with its medial axis (cyan curve)
and bottleneck (red point) at the origin
which has three closest points in $C'$.
The reach is 
$$\tau_{C'} = 
\dfrac{\sqrt[3]{64 - 26\sqrt{2}}}{7}
\approx 0.4298$$
attained at the three points on the end of the medial
axis in the interior of $C'$.  
The weak feature size is $\wfs(C') = \sqrt[3]{\sqrt{2}-1} \approx 0.7454$
attained at the origin.  Hence, $\tau_{C'} < \wfs(C')$.

We note that the exact values for the reach were computed
by using the results of the numerical computation in \Cref{ex:cassiniReach}
together with the exactness recovery method in~\cite{Exactness} 
yielding minimal polynomials of $x^4+2x^2-1$
and $343x^6 - 128x^3 + 8$ for $\tau_C$~and~$\tau_{C'}$, respectively.
Hence, the algebraic degree of the reach is $4$ and $6$, respectively.

\begin{figure}[!t]
    \centering
    \begin{tabular}{ccc}
    \includegraphics[scale=0.12]{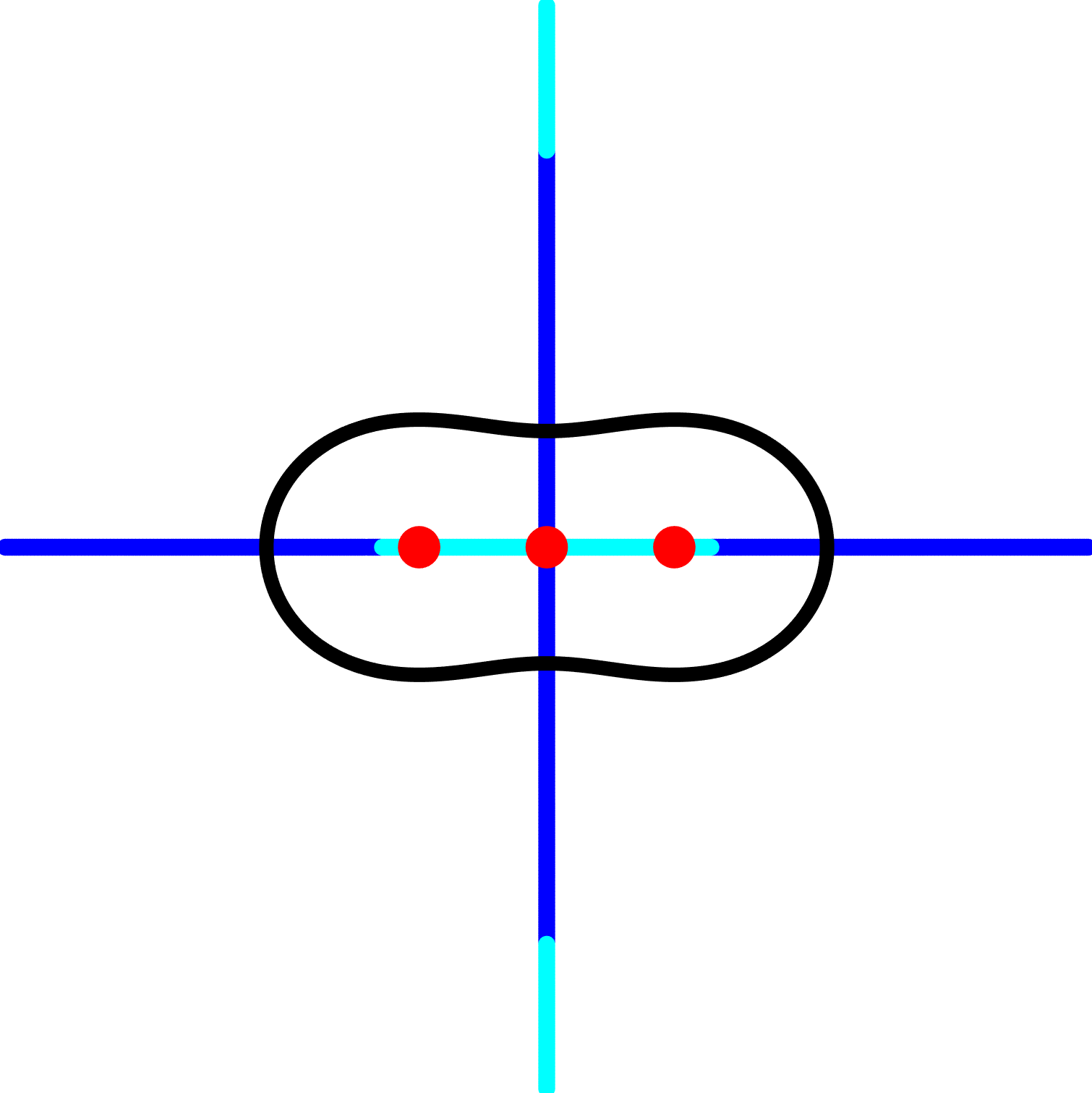} 
     & ~~~~ &
    \includegraphics[scale=0.12]{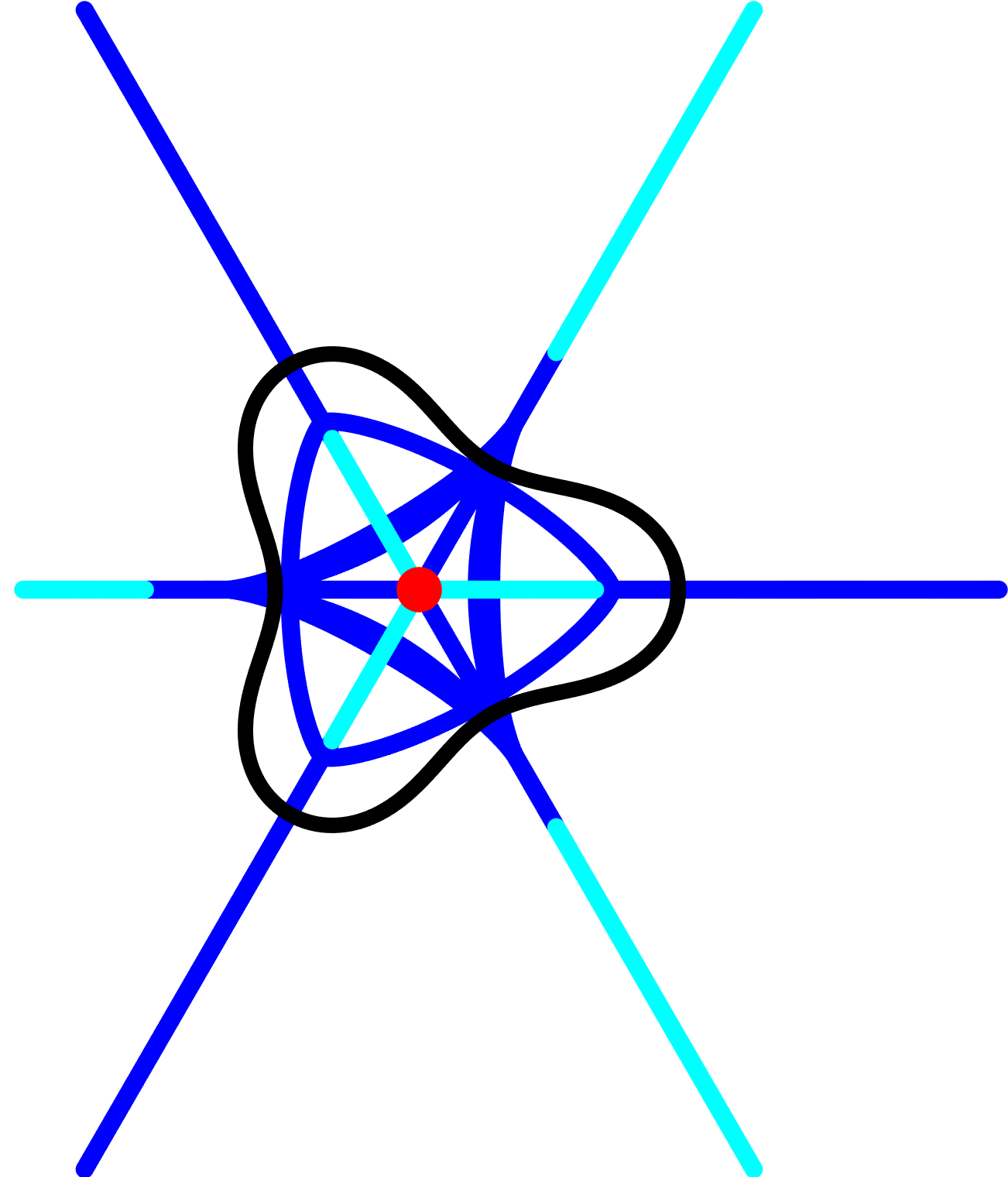} \\
    (a) & & (b) \\
    \end{tabular}
    \caption{Black curve is the Cassini oval with (a) 2 foci 
    and (b) 3 foci.  The red points are
    geometric bottlenecks and the union of the cyan curves 
    form the medial axis.}
    \label{fig:intro_cassinis}
\end{figure}
\end{example}

The following results show that distance functions $d_X:\mr^\numvars\to\mr$ enjoy some properties similar to  Morse functions in Morse theory~\cite{MMorse1934calculus} and justify studying the weak feature size in the algebraic setting. For the sake of analogy, recall that if $f:M\to\mr$ is a Morse function on a compact manifold $M$ with critical points $C_f$ then, by a theorem of A. Morse~\cite{AMorse1939Crit} and Sard~\cite{Sard1942Theorem}, $f(C_f)$ is finite.
By a fundamental theorem of Morse theory (see, e.g.,~\cite[Thm.~3.1]{Milnor1963Morse}), if $f(C_f)\cap [r_1,r_2] = \emptyset$ then $f^{-1}(-\infty,r_1]$ is a deformation retract of $f^{-1}(-\infty,r_2]$. 

\begin{theorem}\label{thm:basic_props}
Let $X$ be a nonempty and compact subset of $\mr^\numvars$, $B$ be the set of geometric bottlenecks of $X$, and $\epsilon$ and $\epsilon'$ satisfy $0 < \epsilon \leq \epsilon'$.
\begin{itemize}
    \item (Grove, 1993 \cite[Prop. 1.8]{Grove93Critical}) If $d_X(B) \cap [\epsilon,\epsilon'] = \emptyset$, then $d_X^{-1}[\epsilon,\epsilon']$ is homeomorphic to $d_X^{-1}(\epsilon)\times [\epsilon,\epsilon']$. In particular, if $\epsilon' < \wfs(X)$, then the thickening $X^\epsilon$ is a deformation retract of $X^{\epsilon'}$. 
    \item (Fu, 1985 \cite[\S 5.3]{FuTubular1985}) If $X$ is semialgebraic, then $d_X(B)$ is finite, and so $\wfs(X) > 0$.
    \item Any compact and semialgebraic set $X$ is a deformation retract of $X^\epsilon$ for some sufficiently small $\epsilon > 0$. We can conclude this in the following way. If $X$ is semialgebraic and bounded, i.e., $X$ is contained in a Euclidean ball of finite radius in $\mr^\numvars$, then $X$ is finitely triangulable \cite[Thm. 3]{lojasiewics1964triangulation}. Since any compact and finitely triangulable set is an absolute neighborhood retract (ANR) \cite[Cor. 3.5]{hanner1951anr}, any compact semialgebraic set $X$ is an ANR. For any compact ANR $X$, $X$ is a deformation retract of $X^\epsilon$ for some sufficiently small $\epsilon > 0$.
\end{itemize}
\end{theorem}

The Morse Lemma implies that the set of critical points of a Morse function on a compact manifold is finite (see, e.g., \cite[Cor. 2.3]{Milnor1963Morse}). 
Nonetheless, the distance-to-$X$ function $d_X$ need not always have 
finitely many geometric bottlenecks even if $X$ is smooth, compact, and an algebraic subset of $\mr^\numvars$ (see, e.g., \Cref{ex:TwoCircles}). We consider this in more detail in \Cref{sec:wfs}.

\subsection{Persistent homology and homology inference} 
In \Cref{sec:sparse}, we will consider a computational application combining algebraic computations for $\lfs$ and $\wfs$ with sampling algorithms (e.g., \cite{di2020sampling,dufresne2019sampling}) and persistent homology to infer the Betti numbers of an algebraic manifold. In this context, the ``persistent homology pipeline'' starts with input in the form of a finite set of points $P\subseteq\mr^\numvars$ and computes the homology of simplicial complexes built from that input.
\begin{definition}\label{def:complexes} Let $P\subseteq\mr^\numvars$ be a finite set of points and $\epsilon\geq 0$.  The \emph{\u{C}ech complex} for $P$ with parameter $\epsilon$, denoted $C_P(\epsilon)$, is the nerve of the set $\{\overline{B}_p(\epsilon)\}_{p\in P}$ where $\overline{B}_x(\epsilon)$ denotes the closed ball of radius $\epsilon$ with center $x$. The \emph{Vietoris-Rips complex}, $R_P(\epsilon)$, is the simplicial complex $\{\sigma\subseteq P \mid \diam(\sigma) \leq \epsilon\}$. See, e.g., \cite[\S 3.2]{HarerBook} for details.
\end{definition} 
In practice, implemented versions of persistent homology (e.g., \cite{ripser,Bauer2021Ripser}) use the Vietoris-Rips complex for computational reasons. For $0\leq\epsilon\leq\epsilon'$, we have subcomplex inclusions  $C_P(\epsilon)\hookrightarrow C_P(\epsilon')$ and  $R_P(\epsilon)\hookrightarrow R_P(\epsilon')$.
It is convenient to assemble these inclusions into functors 
$C_P, R_P:\RR\to\simp$ where $\RR$ denotes the poset of real numbers with the standard ordering and $\simp$ the category of simplicial complexes with simplicial maps as morphisms.  By taking the $\ell^{\rm th}$ homologies with~${\mathbb F}_2$ coefficients of the values of these functors, we obtain \emph{persistence modules} which are functors $H_\ell C_P,H_\ell R_P:\RR\to\vect_{{\mathbb F}_2}$. Each persistence module has an associated \emph{rank function} which summarizes its algebraic structure.
\begin{definition}
Let $M:\RR\to\vect_{\mz/2}$ be a functor and let $\mr^2_< = \{(x,y)\in\mr^2 \mid x < y\}$. 
The \emph{rank function} of $M$ is $\rank(M):\mr^2_<\to\mz$ defined by $\rank(M)(x,y) = \rank(M(x\leq y))$. 
\end{definition}
A standard way to present the information in a rank function is via its \emph{persistence diagram}. Let $\overline{\mr} = \mr \cup \{-\infty,\infty\}$ and let $\overline{\mr}^2_<$ be defined similarly to $\mr^2_<$. 
\begin{definition}
Let $M:\RR\to\vect_{{\mathbb F}_2}$ be a persistence module. The \emph{persistence diagram} of~$M$, if one exists, is the multiset $DM$ in $\overline{\mr}^2_<$ where $\rank(M)(x,y)$ is the number of points in $DM$ strictly above and at least as far left as $(x,y)$ in $\overline{\mr}^2_<$. The multiset $DM$ is called the \emph{persistence diagram} of $M$.
\end{definition}
The persistence modules we have discussed all have persistence diagrams and we assume the same for all persistence modules going forward. Note that this approach to defining persistence diagrams is somewhat non-standard for the sake of brevity. See, e.g., \cite{OudotBook} for a more comprehensive approach. A point $(b,d)$ in $D(H_\ell R_P)$ can be regarded as representing an $\ell$-dimensional homology feature which is ``born'' at radius $b$ and ``dies'' at radius $d$ as shown in~\Cref{fig:pd}.

\begin{figure}
    \centering
    \includegraphics[scale=0.3]{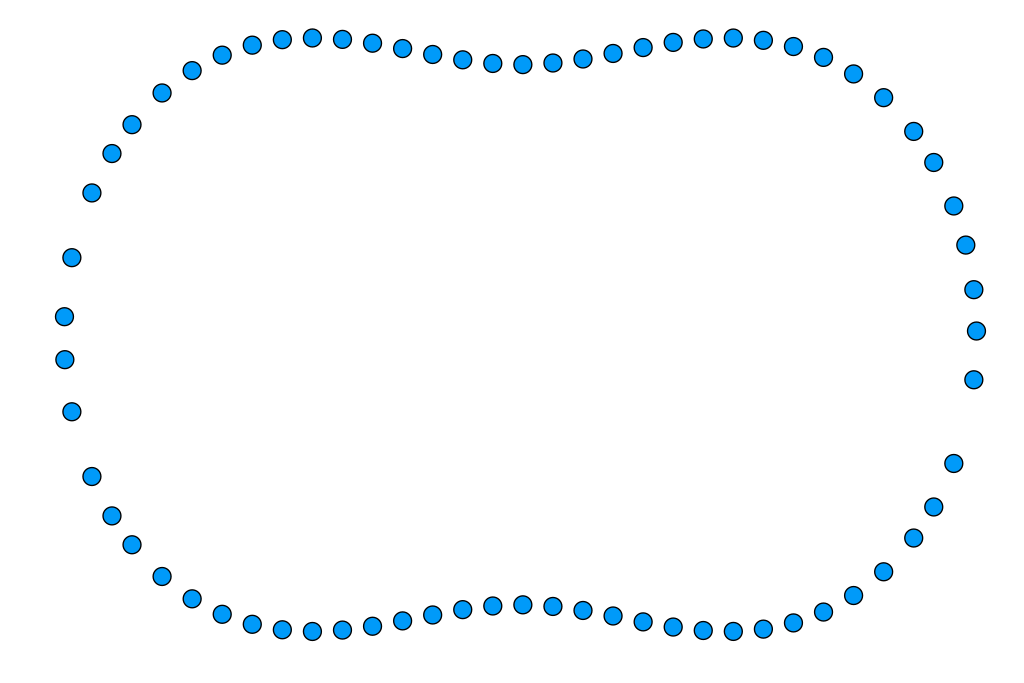}
    \includegraphics[scale=0.45]{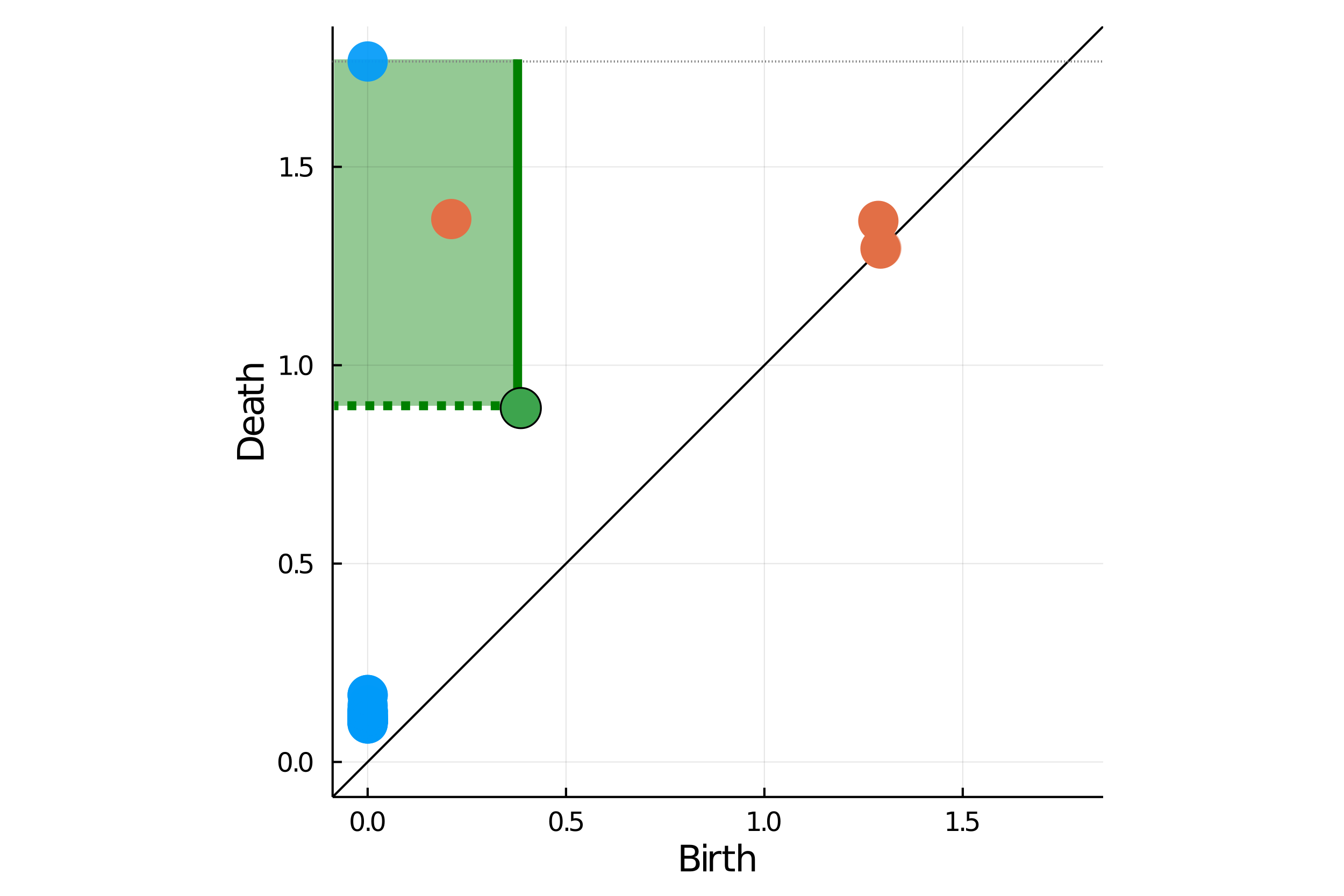}
    \caption{A sample from a curve and its Vietoris-Rips persistence diagram in degree 1 (red) and 0 (blue). The green region depicts inferred homology features as in \Cref{thm:hom_inference}.}
    \label{fig:pd}
\end{figure}

One of the essential theoretical justifications for computing persistent homology is that the Vietoris-Rips persistence diagram of a ``dense enough'' point sample from a compact space $X\subseteq\mr^\numvars$ recovers the Betti numbers of $X$.
\begin{definition}
Let $X,Y \subseteq\mr^\numvars$ be compact and $0\leq \delta\leq\epsilon$. 
The set $X$ is a \emph{$(\delta,\epsilon)$-sample} of~$Y$ if $X \subseteq Y^\delta$ and $X\subseteq Y^\epsilon$.\label{densitydef}
\end{definition} 

\begin{theorem}\cite[Homology Inference Theorem]{chazalwfs,cseh:stability} 
Let $X,\hat{X} \subseteq\mr^\numvars$ such that $X$ 
is compact and semialgebraic and $\hat{X}$ is a $(\delta,\epsilon)$-sample of $X$. If $X^{2(\epsilon+\delta)}$ is homotopy equivalent to $X$, then $\rank(H_\ell C_{\hat{X}})(\epsilon,2\epsilon+\delta)$ is the $\ell^{\rm th}$ Betti number of $X$, $\beta_\ell(X)$. Equivalently, $\beta_\ell(X)$ is the number of points above and to the left of $(\epsilon,2\epsilon+\delta)$ in $D(H_\ell C_{\hat{X}})$. \label{thm:cech_hom_inf}
\end{theorem}

There is also a version of this for Vietoris-Rips complexes.
It follows directly from, e.g.,~\cite[Thm 2.5]{coveragetop}, the Homology Inference Theorem, and an interleaving argument similar to Chazal and Lieutier's proof of the Homology Inference Theorem \cite{chazalwfs}. A proof appears in  \cref{appendixA} for the sake of completeness.

\begin{theorem}\label{thm:hom_inference}
Let $A_\numvars = \sqrt{\frac{2\numvars}{\numvars+1}}$, $\hat{X}$ be a $(\delta,\epsilon)$-sample of a semialgebraic set $X\subseteq \mr^\numvars$, and $\delta' = 2\epsilon(A_n^2 -1)+\delta A_n$. If $2(\epsilon + \delta') < \wfs(X)$ then, for $a(\delta,\epsilon) = 2\epsilon$ and $b(\delta,\epsilon) = 2(2\epsilon A_n + \delta)$, the $\ell^{\rm th}$ Betti number $\beta_\ell(X)$ is the rank of the map obtained from applying $H_\ell$ to the inclusion map $R_{\hat{X}}(a(\delta,\epsilon)) \subseteq R_{\hat{X}}(b(\delta,\epsilon))$.
\end{theorem}

In the algebraic setting, we can construct the following ``homology inference pipeline'' (\Cref{alg:homology_inference}) with persistent homology computations, provided we can compute weak feature sizes and samples of algebraic manifolds. The latter is possible with existing algorithms~\cite{di2020sampling,dufresne2019sampling}.

\begin{algorithm}
    \scriptsize
	\SetKwInOut{input}{Input}\SetKwInOut{output}{Output}\SetKwFunction{Return}{Return}
	\input{A polynomial system $F$ defining a smooth and equidimensional algebraic variety with $X = V(F)\cap\mr^n$ compact and a degree of homology $\ell \geq 0$.}
	\output{A persistence diagram which encodes the $\ell^\text{th}$ $\mathbb{F}_2$-Betti number of $X$ along with additional information.}
    Compute $0 < \omega < \wfs(X)$ \;
    Compute using $\omega$ a $(\delta,\epsilon)$-sample of $X$, $\hat{X}$, which satisfies the assumptions of \Cref{thm:hom_inference} \;
    Compute (e.g., with Ripser \cite{ripser}) and return the Vietoris-Rips persistence diagram of $\hat{X}$ \;
	\caption{\textsc{PH Homology Inference}}
	\label{alg:homology_inference}
\end{algorithm}
\DecMargin{.25em}

\section{Algebraic medial axis, reach, and local feature size}\label{sec:lfs_compute}

The geometric definition of the medial axis and hence the reach and local feature
size in Section~\ref{sec:background} utilize a semialgebraic condition
via closest points.  By replacing closest points with a criticality condition, 
the following provides an algebraic relaxation that is amenable to
computational algebraic geometry over $\CC$.
As before, we assume that \mbox{$X=V(F)\cap\mr^\numvars$} is nonempty and compact 
where $F$ is a polynomial system with real coefficients.

\begin{definition}
Let $F = \{f_1,\dots,f_m\}$ be a system of polynomials in $\numvars$ variables 
with real coefficients such that $V(F)$ is equidimensional and smooth of codimension $c$. 
The \emph{medial axis correspondence} of $F$, denoted $M(F)$, is the algebraic subset of $\CC^\numvars \times \CC^\numvars \times \CC^\numvars$ of points $(x_1,x_2,z)$ which satisfy the equations: 
\[ F(x_1) = F(x_2) = 0, \ \ \ \text{rank}[x_i-z \ JF(x_i)^T]\leq c \  \text{ for } i = 1,2, \ \text{ and } d(x_1, z)^2 = d(x_2, z)^2\]
where $JF(p)$ is the Jacobian matrix of $F$ evaluated at $p$.
If $\Delta$ is the subset of points $(x_1,x_2,z)$ where any two of the entries are equal,
the \emph{algebraic medial axis of $F$} is the 
closure of the image of the projection of $M(F)\setminus\Delta$ onto its third factor.
\label{def:alg_medial}
\end{definition}

In particular, the condition $\text{rank}[x_i-z \ JF(x_i)^T]\leq c$
enforces that $d(x_i,z)^2$ is critical for $x_i\in V(F)$.  
Therefore, it is clear that the algebraic medial axis contains the medial axis~$\medial_X$.

As with the medial axis, one can consider subsets of the algebraic medial axis based
on the number of equidistant critical points.  That is, for $k\geq2$, 
the \emph{$k$-medial axis correspondence of $F$}, denoted $M_k(F)$, is the 
algebraic subset of $\left(\CC^\numvars\right)^k\times \CC^\numvars$ of points
$(x_1,\dots,x_k,z)$ which satisfy the equations:
\[ \begin{array}{cc} 
F(x_i) = 0 & \multirow{2}{*}{$\text{ for } i = 1,\dots,k$} \\
\text{rank}[x_i-z \ JF(x_i)^T]\leq c &
\end{array}
\text{ and } d(x_1, z)^2 = d(x_j, z)^2 \  \text{ for } j = 2,\dots,k. \]
If $\Delta$ is the subset of points $(x_1,\dots,x_k,z)$ where any two of the entries are equal,
the \emph{algebraic $k$-medial axis of $F$} is the 
closure of the image of the projection of $M_k(F)\setminus\Delta$ onto its last~factor.
In particular, the algebraic medial axis is the algebraic $2$-medial axis.

\begin{example}\label{ex:cassini2algmedial}
For the Cassini oval with 2-foci in \Cref{ex:cassini_def}, the algebraic medial axis
is the union of the coordinate axes.  
Moreover, for both the Cassini oval with 2- and 3-foci in \Cref{fig:intro_cassinis}, 
the cyan curves form the medial axis and the union
of the blue and cyan curves form the algebraic medial axis.
\end{example}

\begin{example}\label{ex:AlgMedialAxisGeneral}
The algebraic medial axis of a general plane curve of degree $d\geq 2$ is also a plane curve.
After randomly selecting coefficients, we used
{\tt Bertini}~\cite{bertini} to compute the degree of the algebraic medial
axis for $2\leq d \leq 9$ as shown in the following table:
$$\begin{array}{c|c}
d & \hbox{degree of algebraic medial axis} \\
\hline
2 & 2 \\
3 & 30 \\
4 & 120 \\
5 & 320 \\
6 & 690 \\
7 & 1302 \\
8 & 2240 \\
9 & 3600 
\end{array}$$
In particular, for $2\leq d \leq 9$, the degree of the 
algebraic medial axis for a general plane curve of degree $d$
is 
$$\binom{d}{2}(d^2+3d-8) = \frac{d(d-1)(d^2+3d-8)}{2}$$
and we conjecture that this formula holds for all $d\geq2$.
\end{example}

We can investigate the reach and local feature size by considering optimization problems on $M(F)$. The solution to $\min\{d(x_1,z)^2 \mid (x_1,x_2,z)\in M(F)\setminus\Delta , d(x_1,z)^2 > 0 \}$, for instance, is a lower bound on the reach. By using first-order critical conditions on $M(F)$, i.e. Lagrange multipliers, we can define critical conditions for the reach and and local feature size.

\begin{definition}
Let $F = \{f_1,\dots,f_m\}$ be a system of polynomials in $\numvars$ variables 
with real coefficients such that $V(F)$ is equidimensional and smooth of codimension $c$. 
The \emph{critical reach correspondence} of $F$, denoted $C(F)$, 
corresponds with first-order critical conditions of
$d(x_1,z)^2$ on $M(F)$.  

Additionally, for $w\in\CC^n$, the \emph{critical local feature size correspondence} of $F$ with respect to $w$, denoted $L(F,w)$, corresponds
with first-order critical conditions of $d(w,z)^2$ on $M(F)$.
\end{definition}

Since the reach and local feature size are defined in terms of minimality conditions,
they are captured in the critical reach and critical local feature size correspondences, respectively.

In order to write down explicit equations, a choice needs to be made
on how to enforce rank and first-order criticality conditions.  
As an illustration, consider the case when $m = c$
using a null space approach, e.g., see \cite{RankDefSet}.
Since $V(F)$ is smooth of codimension $c$, for $i=1,2$, 
$\text{rank}[x_i-z \ JF(x_i)^T]\leq c$ is true if and only if 
$$x_i-z + JF(x_i)^T \lambda_{i} = 0$$
for some $\lambda_i\in\CC^c$.  
Hence, an alternative formulation for the medial axis correspondence to that in \Cref{def:alg_medial} consists of $(x_1,x_2,z,\lambda_1,\lambda_2)$ where
\begin{equation}\label{eq:MedialAxisCor}
F(x_1) = F(x_2) = 0, \ \ \ x_i-z + JF(x_i)^T \lambda_{i} = 0 \  \text{ for } i = 1,2, \ \text{ and } d(x_1, z)^2 = d(x_2, z)^2\text{.} 
\end{equation}
This system, say $F_M$, consists of $2c+2n+1$ equations in $3n+2c$ variables. After removing~$\Delta$ and projecting, one expects the algebraic medial axis to be a hypersurface in $\CC^n$.  

Let $\delta_0\in\CC$ and $\delta_1\in\CC^{2n+2c+1}$ and treat
$[\delta_0,\delta_1]\in\PP^{2n+2c+1}$.  Then, the critical reach correspondence $C(F)$ 
on \eqref{eq:MedialAxisCor} corresponds with the set of points $(x_1,x_2,z,\lambda_1,\lambda_2,[\delta_0,\delta_1])$ with
$$\nabla\left(d(x_1,z)^2\right)^T \delta_0 + (J(F_M)(x_1,x_2,z,\lambda_1,\lambda_2))^T\delta_1 = 0, F_M(x_1,x_2,z,\lambda_1,\lambda_2) = 0 $$
which consists of a well-constrained system consisting of $4c+5n+1$ equations.

Similarly, the critical conditions $L(F,w)$ for $d(w,z)^2$
on \eqref{eq:MedialAxisCor} correspond with
$$\nabla\left(d(w,z)^2\right)^T \delta_0 + (J(F_M)(x_1,x_2,z,\lambda_1,\lambda_2))^T\delta_1 = 0, F_M(x_1,x_2,z,\lambda_1,\lambda_2) = 0 $$

The following characterizes components of these critical correspondences.

\begin{theorem}\label{thm:reach_and_lfs}
Let $F$ be a polynomial system such that $V(F)$ is smooth and equidimensional of codimension $c$.
\begin{itemize}
    \item[(a)] Let $D:M(F)\to\CC$ be defined by $(x_1,x_2,z)\mapsto d(x_1,z)^2$. 
Then, $D$ is constant on every connected component of $C(F)$ with projection onto $M(F)$ not contained in $\Delta$.
    \item[(b)] Fix $w\in\CC^\numvars$ and let $D_w:M(F)\to\CC$ be defined by $(x_1,x_2,z)\mapsto d(w,z)^2$. Then, $D_w$ is constant on every connected component of $L(F,w)$ with projection onto $M(F)$ not contained in $\Delta$.    
\end{itemize}
\end{theorem}
\begin{proof}
We prove (a) and omit a similar proof of statement (b).
The fact that $D$ is constant on every irreducible component of $C(F)$ with projection not contained in $\Delta$ follows directly from the construction and the algebraic version of Sard's Theorem, e.g., see \cite[Thm~A.4.10]{somnag}. Since irreducible components are connected, each connected component must be the union of irreducible components. Furthermore, any irreducible component which is not a connected component must intersect at least one other distinct irreducible component. Thus, the constancy of $D$ can be extended to connected components yielding (a).
\end{proof}

Let $\mathcal{C}(F)$ denote the union of connected components of $C(F)$ with projection onto $M(F)$ not contained in $\Delta$ and similarly for $\mathcal{L}(F,w)$. Clearly, the reach is a value of $d(x_1,z)^2$ on~$\mathcal{C}(F)$.
In fact, it is the minimum positive critical value of $d(x_1,z)^2$ on $\mathcal{C}(F)$
for which there is a real point that attains that critical value. 
Since there can only be finitely many critical values, 
this immediately provides an approach to compute the reach as follows.
First, one computes a finite set of points that contains at least
one point in each connected component of $\mathcal{C}$.
Then, one evaluates $d(x_1,z)^2$ on the finite set of points to obtain the finite set of critical values.  
Immediately from this algebraic computation, 
one has that the minimum of the positive critical values is 
a lower bound on the reach.  To obtain the actual value of the reach, 
one would need to employ an additional reality test, e.g., \cite{RealPoints}, to 
test for the existence of real points on the corresponding connected components.
By searching in an increasing order starting with the minimum positive critical value,
the reach is determined when a real point exists on the corresponding connected components.

Using numerical algebraic geometry, e.g., see \cite{bertinibook,somnag}, 
there are several approaches using homotopy continuation that can be used
to compute a finite set of points containing at least one on each connected component.
For example, parameter homotopies \cite{CoeffParam} can be used to provide such a set.
By looking at a finer decomposition based on irreducibility rather than connectedness,
one can compute a finite set of points containing at least one on each irreducible component
using a first-order general homotopy~\cite{FirstOrderGen}.   
Another approach is to utilize a sequence of homotopies 
based on using linear slicing via a cascade \cite{Cascade} or regenerative
cascade \cite{RegenCascade}.
This last approach actually computes witness point sets (see~\cite{bertinibook,somnag}
for more details) which can then be used directly for reality testing
via~\cite{RealPoints} when one expects positive-dimensional components.
When the set of critical points is finite, all approaches yield
the entire set of critical points and reality testing simply decides
the reality of each~critical~point.

A similar argument follows
for the local feature size as well.  Moreover, one can treat $w$ as a parameter and utilize a parameter homotopy \cite{CoeffParam} to perform this computation efficiently at many different points. We summarize this in the following.

\begin{corollary}\label{cor:reach_and_lfs_hom}
Let $F$ be a polynomial system in $\numvars$ variables with real coefficients such that $V(F)$ 
is smooth and equidimensional of codimension $c$
and $X=V(F)\cap\mr^\numvars$ is nonempty and compact.  Fix $w\in\mr^\numvars$ and let $D$
and $D_w$ be as in \Cref{thm:reach_and_lfs}.
\begin{itemize}
    \item[(a)] Using a parameter homotopy~\cite{CoeffParam}, one can compute a finite set of points $S$
    which contains at least one point in each connected component of $\mathcal{C}(F)$.  Then,
\begin{equation}\label{eq:ReachLowerBound}
    0<\min_{s\in S \text{~with~} D(s) > 0} \sqrt{D(s)} \leq \tau_X.
    \end{equation}
    \item[(b)] Using a parameter homotopy~\cite{CoeffParam}, one can compute a finite set of points $S_w$
    which contains at least one point in each connected component of $\mathcal{L}(F,w)$.  Then,
    \begin{equation}\label{eq:LFSLowerBound}
    0<\min_{s\in S_w \text{~with~} D_w(s) > 0} \sqrt{D_w(s)} \leq \lfs(w).
    \end{equation}
\end{itemize}
\end{corollary}

This section concludes with some illustrative examples.

\begin{example}\label{ex:cassiniReach}
Consider computing the reach for the Cassini ovals with 2 and 3 foci from 
\Cref{ex:cassini_def}.  
We utilized a parameter homotopy in the corresponding space of multihomogeneous systems.
For the Cassini oval with 2 foci, 
the lower bound in \eqref{eq:ReachLowerBound} is approximately $0.6436$ which is attained at three different critical points computed by the homotopy.
As shown in \Cref{fig:Cassini3reach}(a), all three are real and thus 
the lower bound in \eqref{eq:ReachLowerBound} is equal to the reach.

For the Cassini oval with 3 foci, 
the lower bound in \eqref{eq:ReachLowerBound} is approximately $0.3611$.
Since this arises from nonreal isolated solutions to the critical point system,
this can easily be discarded as not being equal to the reach.
The next two smallest positive critical values are approximately $0.3674$ and 
$0.3868$ which also arise from nonreal isolated solutions to the critical point system
and thus can be discarded as not being equal to the reach.
Finally, the fourth smallest positive critical value is approximately $0.4298$
which does arise from real solutions to the critical point system and is thus
equal to the reach.  The reach attaining points are shown
in \Cref{fig:Cassini3reach}(b).

As remarked in \Cref{ex:cassini_def}, the minimal polynomial for the reach
of the Cassini oval with 3 foci is $343x^6-128x^3+8$.  Since the
critical reach correspondence is defined over the rational numbers,
each root of this minimal polynomial is also a critical value.
\Cref{fig:Cassini3reach}(c) shows the critical points 
associated with the other real root which is approximately $0.6648$.

\begin{figure}
    \centering
    \begin{tabular}{ccccc}
    \includegraphics[scale=0.4]{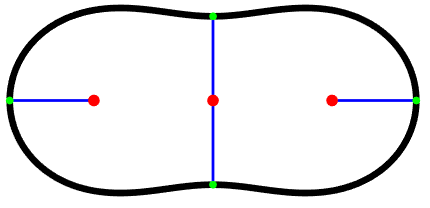} 
    & ~~~~ &
    \includegraphics[scale=0.4]{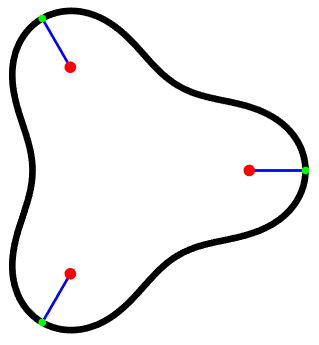} 
    & ~~~~ &
    \includegraphics[scale=0.4]{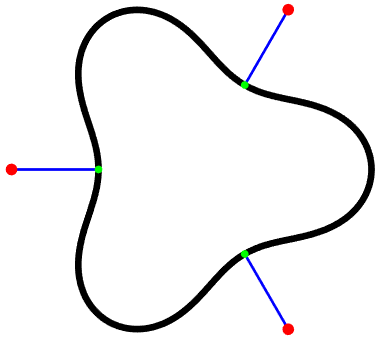} \\
    (a) & & (b) && (c)
    \end{tabular}
    \caption{Reach attaining points for the Cassini oval with (a) 2 foci and (b) 3 foci.
    (c)~Critical points arising from the algebraic closure of the reach for
    the Cassini oval with 3 foci.
    }
    \label{fig:Cassini3reach}
\end{figure}
\end{example}

\begin{example}\label{ex:UnitCircle}
The medial axis of the unit circle defined by $x_1^2+x_2^2=1$
is the origin and thus the reach of $1$ is attained at the origin.
However, this reach attaining point 
is not isolated with respect to the critical reach correspondence
since there are infinitely-many points 
on the unit circle where the reach is attained.
Hence, the corresponding points computed 
via homotopy continuation need not be real.
For example, using a multihomogeneous
homotopy, $1$ is the unique positive critical value
arising from $78$ distinct endpoints,
none of which correspond with real points on the unit circle.
Since $1$ is the only positive critical value, it is the reach.
\end{example}

\begin{example}\label{ex:TwoCircles}
The medial axis of the union of two concentric circles defined
by 
$$(x_1^2 + x_2^2 - 1)(x_1^2 + x_2^2 - 9)=0$$
is the union of the origin and the circle centered at the origin of radius $2$.
Thus, the reach is $1$ which is attained at every point on the medial axis
as shown in \Cref{fig:NestedCircles}.
For example, using a multihomogeneous
homotopy, the minimum positive critical value is $1$
which arises from 184 distinct endpoints.
Of these, 168 correspond with the origin
while the other 16 correspond with distinct points in $\CC^2$
satisfying $x_1^2+x_2^2=4$.  From \eqref{eq:ReachLowerBound}, 
this one homotopy shows that the reach is at least $1$.  
For this example, it is easy to verify that there exist real critical points
that yield a critical value of $1$ which shows that the reach is indeed equal to~$1$.

\begin{figure}
    \centering
    \includegraphics[scale=0.2]{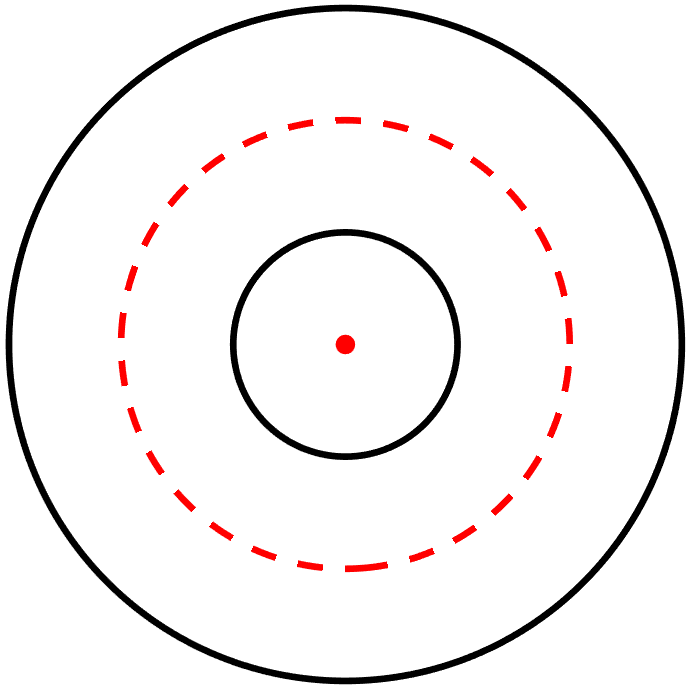}
    \caption{Reach obtained on a curve and a point for two concentric circles}
    \label{fig:NestedCircles}
\end{figure}
\end{example}

\section{Bottlenecks and weak feature size} 
\label{sec:wfs}

The following expands upon the definition of geometric bottlenecks from \Cref{def:wfs}
and considers successive approximations of the weak feature size using {\it higher order bottlenecks.}

\begin{definition}
\label{def:geom_bottleneck}
Let $X$ be a compact subset of $\mr^\numvars$. A geometric bottleneck $z$ in $\crit(d_X)$ has \emph{order} $k\geq 2$ if $z$ is a convex combination of $k$ affinely independent points in $\pi_X(z)$ and is not a convex combination of any fewer number of points in $\pi_X(z)$. We will often refer to such a point $z$ as a \emph{geometric k-bottleneck} of $X$. 
\end{definition}
\begin{remark}
\Cref{def:geom_bottleneck} resembles a generalization of the index of critical points of a Morse function introduced by Gershokovich and Rubinstein \cite{Gersh1997Min}.
The treatment by Bobrowski and Adler
renders this connection clearer for distance functions \cite[Def.~2.1]{Bobrowski2014Random}, albeit for the case where $X$ is a finite point set. A geometric $k$-bottleneck of $X$ is a critical point of $d_X$ with index $k-1$ using that terminology.  When $X$ is a finite set of points, this notion of index yields a decomposition similar to the classic cellular decomposition theorem of Morse theory (see, e.g., \cite[Thm.~3.5]{Milnor1963Morse} and \cite[\S4.2]{Bobrowski2014Random}). 
This does not extend to the case when $X$ is not finite. 
In particular, the Cassini oval with 2 foci in \Cref{ex:cassini_def} 
and the unit circle in \Cref{ex:UnitCircle} are both counter examples. 
We use the term \emph{order} rather than \emph{index} to clarify that a Morse-type result 
does not apply in our setting.
\end{remark}

Before considering the algebraic setting, the following highlights the relationship between
geometric $k$-bottlenecks and weak feature size from \Cref{def:wfs}.

\begin{proposition}\label{prop:fin_many_order}
If $X$ is a compact subset of $\mr^\numvars$, then every geometric bottleneck has order
at most $\numvars+1$ and
\[{\rm wfs}(X)=\inf_{z\text{ geom. }k\text{-bottleneck, }2\leq k\leq\numvars+1}d_X(z).\]
\end{proposition}
\begin{proof}
Suppose that $z$ is a geometric bottleneck of $X$.
Then, by definition, $z$ is in the convex hull of $\pi_X(z)$. 
By Carath\'{e}odory's Theorem \cite{caratheodory1911}, $z$ is a convex combination of at most $\numvars+1$ points in $\pi_X(z)$
which shows that the order of $z$ is at most $\numvars + 1$.
\end{proof}

\begin{remark}
From \Cref{prop:fin_many_order}, it is natural to ask, for algebraic manifolds, if 
one must use all possible orders of geometric bottlenecks to determine the weak feature size
or if one could use less, e.g., use only geometric $2$-bottlenecks.
The Cassini oval with 3 foci in \Cref{ex:cassini_def} lies in $\mr^2$
and has no geometric $2$-bottlenecks.  In particular, the
weak feature size is attained at the origin, which is a geometric bottleneck of maximal order $3$.
Similar Cassini oval constructions generalize to higher dimensions
and also generalize \cite[Ex 3.4]{di2020sampling}.
\end{remark}

Following a similar approach as in Section~\ref{sec:lfs_compute}, one can relax the conditions 
of a geometric $k$-bottleneck to obtain algebraic conditions amenable to computational algebraic
geometry over $\CC$.  As before, we assume that $X=V(F)\cap\mr^\numvars$ is
nonempty and compact where $F$ is a polynomial system with real coefficients.

\begin{definition}\label{def:locus}
Let $F = \{f_1,\dots,f_m\}$ be a system of polynomials in $\numvars$ variables 
with real coefficients such that $V(F)$ is equidimensional and smooth of codimension $c$
and $k\geq 2$.  The \emph{$k^{\rm th}$ bottleneck correspondence} of $F$, denoted $B_k(F)$, is
\[ \left\{(x_1,\dots,x_k,t_1,\dots,t_k)\in \left(\mc^{\numvars}\right)^k \times \mc^k~\left| \begin{array}{l}
 \begin{array}{l} \sum_{i=1}^k t_i = 1, \end{array} \\
 \begin{array}{l} z = \sum_{i=1}^k t_ix_i, \end{array}\\ 
 \begin{array}{l} d(x_1, z)^2 = d(x_j, z)^2 \  \text{ for } j = 2,\dots,k \end{array}\\
\begin{array}{ll}
F(x_i) = 0 & \multirow{2}{*}{\text{ for } $i = 1,\dots,k$.} \\ 
\text{rank}[x_i-z \ JF(x_i)^T]\leq c
\end{array}
\end{array}\right\}\right..
\] 
Let $\Gamma_k\subset \left(\mc^\numvars\right)^k \times \mc^k$ consist of all points
$(x_1,\dots,x_k,t_1,\dots,t_k)$ where any $t_i$ is $0$ or the set $\{x_1,\dots,x_k\}$ is affinely dependent.  Consider the map $\rho_k:\left(\mc^\numvars\right)^k\times\mc^k\to\mc^\numvars$ defined by $\rho_k(x_1,\dots,x_k,t_1,\dots,t_k)=\sum_{i=1}^k t_ix_i$.
A point $z\in\CC^\numvars$ is an \emph{algebraic $k$-bottleneck} of $V(F)$ if $z \in \rho_k(B_k(F) \setminus \Gamma_k)$.  A \emph{real algebraic $k$-bottleneck} of $V(F)$ is a point in $\mr^\numvars$ 
which is an algebraic $k$-bottleneck.  Let $X = V(F)\cap\mr^\numvars$
and $R_{X,k} = X^k \times (0,1)^k \subset \left(\mr^\numvars\right)^k \times \mr^k$.
A \emph{real algebraic $k$-bottleneck} of $X$ is a point in $\mr^\numvars$
in the image of 
$\rho_k( (B_k(F) \cap R_{X,k}) \setminus \Gamma_k)$.
\end{definition}

\begin{remark}
Following the notation of \Cref{def:locus}, every geometric $k$-bottleneck of $X$
is a real algebraic $k$-bottleneck of $X$.  In particular, one has the following relationship:
\[
\begin{array}{c}
{\small 
\{\text{geometric $k$-bottlenecks of } X\} \subseteq \{\text{real algebraic $k$-bottlenecks of } X\}
\subseteq } \\
{\small \{\text{real algebraic $k$-bottlenecks of } V(F)\} \subseteq \{\text{algebraic $k$-bottlenecks of } V(F)\} = \rho(B_k(F)\setminus \Gamma_k).}
\end{array}
\]
Typically, these inclusions are strict as the examples in \Cref{sec:examples} exhibit.
In particular, for the second inclusion, 
it is possible for the image of $\rho_k$ to be real for~nonreal~input.
\end{remark}

\begin{example}\label{ex:ellipsoidPerturb}
Consider computing the algebraic $2$-bottlenecks for 
the perturbed ellipsoid $X\subseteq\mr^2$ 
from the Introduction defined by \mbox{$F = x_1^2 + x_2^2 + x_3^2/2 + x_1x_3/7 - 1$}.
The set $B_2(F)\setminus\Gamma_2$ consists of three points up to symmetry, which are depicted in \Cref{fig:ellipsoid2}
with the black segments corresponding to the geometric $2$-bottleneck of~$X$
while the blue segments correspond with real algebraic $2$-bottlenecks of~$X$
that are not geometric $2$-bottlenecks of $X$.  
\end{example}

\subsection{Critical values}

The following exhibits a result similar to \Cref{thm:reach_and_lfs}.

\begin{theorem}\label{thm:fin_many_crit_val}
Let $F$ be a polynomial system such that $V(F)$ is smooth and equidimensional of
codimension $c$.  Let $k\geq 2$ and $D_k:B_k(F)\to\mc$ be defined by
$$(x_1,\dots,x_k,t_1,\dots,t_k)\mapsto \frac{1}{k}\sum_{i=1}^k d\left(x_i,\rho_k(x_1,\dots,x_k,t_1,\dots,t_k)\right)^2.$$
Then, $D_k$ is constant on every connected component of $B_k(F)\setminus\Gamma_k$.
\end{theorem}
\begin{proof}
Similarly to \Cref{thm:reach_and_lfs}, it suffices to show 
constancy for an irreducible component~$C$ 
not contained in $B_k(F)\setminus \Gamma_k$.

The following shows that every point in $B_k(F)\setminus\Gamma_k$
is a critical point of $D_k$. Since $D_k\vert_{C}$ is an algebraic map of irreducible quasiprojective algebraic sets, it follows by the algebraic version of Sard's Theorem, e.g.,~\cite[Thm~A.4.10]{somnag}, 
that $D_k$ is not dominant, and therefore $D_k(C)$ is a single point because otherwise~$C$ is 
not irreducible.  Note that the critical points of $D_k$ are the same as those of $k\cdot D_k$
so we will consider $k\cdot D_k$ for simplicity.  

Let $A_k(F)\subset\left(\mc^\numvars\right)^k\times\mc^n$ be 
the set of $(x_1,\dots,x_k,z)$ satisfying
\[
 \begin{array}{l}
  F(x_i) = 0 \ \text{ for } i = 1,\dots,k, \\
  d(x_1, z)^2 = d(x_j, z)^2 \  \text{ for } 2 \leq j \leq k.
\end{array}
\] 
Clearly, there is an inclusion map $i:B_k(F)\to A_k(F)$ 
given by \[ (x_1,\dots,x_k,t_1,\dots,t_k)\mapsto \left(x_1,\dots,x_k,\sum_{i=1}^k t_i x_i\right).\] 
By the chain rule, we need only prove that any point in the image of $i$ is a critical point of
the map
$D_k':A_k(F)\to\mc$ defined by $(x_1,x_2,\dots,x_k,z) \mapsto \sum_{i=1}^k d(x_i,z)^2$. 
Since~$V(F)$ is smooth and equidimensional, one may check directly that $A_k(F)$ has codimension $kc + k - 1$. By elementary row operations, one reduces the problem to showing that $(x_1,\dots,x_k,z)$ is a critical point of $D_k'$ 
if the $(km + k)\times (k\numvars + \numvars)$ matrix 

\[
\begin{pmatrix} 
JF(x_1) & 0 & 0 & 0 & 0 & \dots & 0 \\
0 & JF(x_2) & 0 & 0 & 0 & \dots & 0 \\ 
  &         &   & \vdots &   & &  \\
0 & 0       & 0 & \dots & 0 & JF(x_k) & 0 \\ 
0 & -(x_2 - z)^T & 0 & \dots & 0 & 0  & (x_2 - x_1)^T \\ 
0 & 0 & -(x_3 - z)^T & 0 & \dots & 0 &  (x_3 - x_1)^T \\ 
 &  & & \vdots &  &  & \\
0 & 0 & 0 & \dots & 0 & -(x_k-z)^T &  (x_k-x_1)^T \\ 
(x_1 - z)^T & 0 & 0 & 0 & \dots & 0 & (z-x_1)^T
\end{pmatrix}
\]
has rank at most $kc + k - 1$
where $m$ is the number of polynomials in $F$.
Suppose that \mbox{$(x_1,\dots,x_k,z)$} is in the image of the inclusion map $i$.
Then, the first $k\numvars$ columns of this matrix contribute at most $kc$ to the dimension of the column space and the final~$\numvars$ columns contribute at most by $k-1$ since $x_2-x_1,\dots,x_k-x_1$ span the affine hull of $x_1,\dots,x_k$ and $z-x_1$ is in that affine hull. 
Altogether the rank of the matrix is at most $kc + k - 1$.  
\end{proof}

\begin{remark}
For $k = 2$, this proof shows that the algebraic $2$-bottlenecks of $V(F)$ correspond
with a subset of the Zariski closure $C(F)\setminus \Delta$ where $C(F)$ is the critical reach correspondence.  In contrast to $C(F)$, however, the correspondence $B_2(F)$ does not contain functions corresponding to the gradients of $\rank$ conditions.
This is illustrated in 
\Cref{fig:Cassini3reach}(a) for the Cassini oval with $2$ foci.
\end{remark}

As with \Cref{thm:reach_and_lfs} yielding \Cref{cor:reach_and_lfs_hom}, 
\Cref{thm:fin_many_crit_val} provides the following.

\begin{corollary}\label{cor:wfs_pd_cor}
Let $F$ be a polynomial system in $\numvars$ variables with real coefficients such that $V(F)$ 
is smooth and equidimensional of codimension $c$
and $X=V(F)\cap\mr^\numvars$ is nonempty and compact.  
For each $k=2,\dots,\numvars+1$, 
one can use a parameter homotopy~\cite{CoeffParam} to compute a finite set of points $E_k$
    which contains at least one point in each connected component of $B_k(F)\setminus\Gamma_k$.  Then,
\begin{equation}\label{eq:wfs_lower_bound}
    0< \min_{k=2,\dots,\numvars+1}\left(\min_{e\in E_k \text{~with~} D_k(e) > 0} \sqrt{D_k(e)}\right) \leq \wfs(X).
    \end{equation}
\end{corollary}

\begin{remark}
As with \Cref{cor:reach_and_lfs_hom}, an additional reality test, e.g., \cite{RealPoints},
can be used to determine the weak feature size.  When $\bigcup_{k=2}^\numvars \left(B_k(F)\setminus\Gamma_k\right)$
is finite, reality testing simply decides the reality of each critical point.
\end{remark}

\begin{remark}
Compactness may be removed as a requirement in \Cref{thm:fin_many_crit_val} and \Cref{cor:wfs_pd_cor}, but some care is necessary when $X = V(F)\cap\mr^\numvars$ is not compact. 
As an illustration, consider $F = x^2y^2-1$ with $X$ shown in \Cref{fig:hyperbola}.
Then, \Cref{thm:fin_many_crit_val} shows that the weak feature size of $X = V(F)\cap\mr^2$ inside any closed Euclidean ball of finite radius centered at the origin that intersects $X$ in~$\mathbb{R}^2$ must be positive.  By an explicit computation, one can see that the only contributors to the weak feature size in $B_2(F)\setminus\Gamma_2$ are  isolated solutions as shown in \Cref{fig:hyperbola}. 
The subtlety is that the manifold $V(F)\cap\mr^2$ is not homotopy equivalent to any of its thickenings and thus it is not an absolute neighborhood retract.
Therefore, \Cref{thm:basic_props} does not apply. 

\begin{figure}
    \centering
    \includegraphics[scale=0.22]{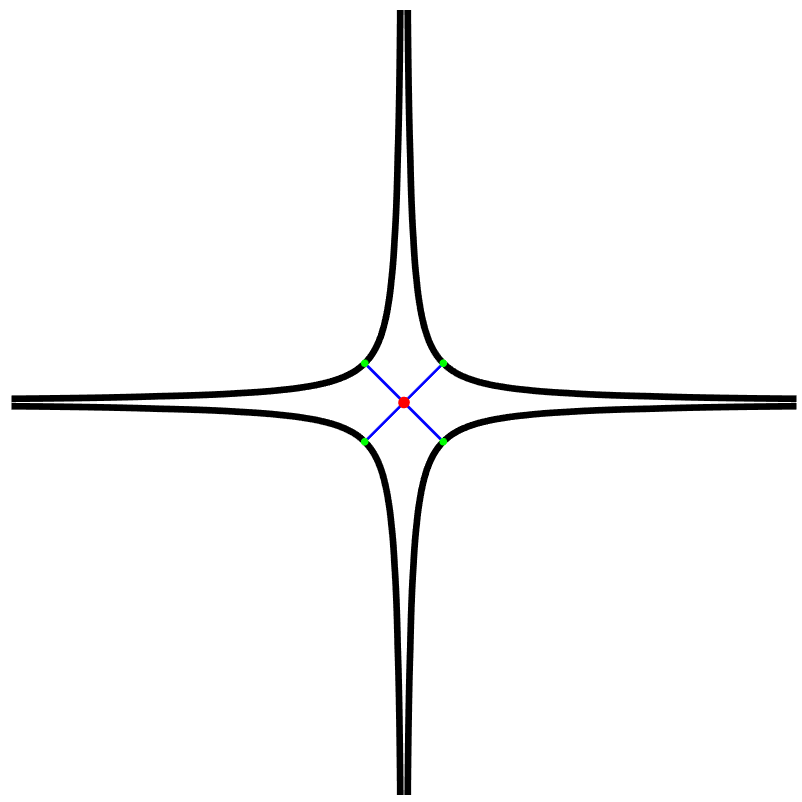}
    \caption{Real solutions of $x^2y^2=1$ with its geometric bottlenecks}\label{fig:hyperbola}
\end{figure}
\end{remark}

\begin{example}\label{ex:ellipsoid}
Consider the ellipsoid $X\subseteq\mr^2$ defined by $F = x_1^2 + x_2^2 + x_3^2/2 - 1$ from the Introduction
that is depicted in \Cref{fig:ellipsoid}.
We consider computing the algebraic $2$-bottlenecks using 
\Cref{cor:wfs_pd_cor} in two different ways: using a parameter homotopy
in the corresponding space of multihomogeneous systems
and using a parameter homotopy from the perturbed ellipsoid
computed in \Cref{ex:ellipsoidPerturb}.  

For the first approach, one obtains $6$ points in $B_2(F)\setminus\Gamma_2$.
Two of these points are real and equal up to symmetry. 
They correspond with the
blue segments connecting the magenta points in \Cref{fig:ellipsoid}
at a distance of $\sqrt{2}$ from the origin, i.e., $D_2 = 2$.
The other 4 points are nonreal and have $D_2 = 1$.  
These lie on a positive-dimensional component arising from antipodal points
on the unit circle in the $(x_1,x_2)$-plane whose real points are shown in \Cref{fig:ellipsoid}.
For this example, it is easy to verify that there exist 
real points on this component which are also geometric $2$-bottlenecks.

For the second approach, we can consider the family of algebraic $2$-bottlenecks
for 
$$F_t = x_1^2+x_2^2+x_3^2/2+tx_1x_2-1.$$ \Cref{ex:ellipsoidPerturb}
shows that, at the generic parameter value $t = 1/7$, $V(F_t)$ has three
algebraic $2$-bottlenecks.  We then used a parameter homotopy to track these three solutions along
the sufficiently general path defined by
$$t(s) = \frac{1}{7}\cdot \frac{\gamma s}{1-s+\gamma s} \text{~~where~~} \gamma = 2 + 3\sqrt{-1}$$
as $s$ goes from $1$ to $0$.
This yielded three real solutions which lie along the three coordinate axes.
The two lying along the $x_1$ and $x_2$ coordinate axes have $D_2=1$
while the third that lies along the $x_3$ coordinate axis has $D_2=2$.
\end{example}

\subsection{Critical points}
The remainder of this section considers the finiteness 
of algebraic $k$-bottlenecks
for general complete intersections of codimension $c$ in $\CC^\numvars$.
Of course, we naturally assume that $\numvars\geq 1$ and $1\leq c\leq \numvars$.
Let $(d_1,\dots,d_c) \in \NN^c$ and 
consider $P_i=\PP^{N_i}$ where $N_i=\binom{d_i+\numvars}{d_i}-1$ which is the parameter
space of hypersurfaces in $\CC^\numvars$ of degree at most $d_i$. 
Furthermore, complete
intersections in $\CC^\numvars$ of codimension $c$ and degree type
$(d_1,\dots,d_c)$ are parameterized by an open subset $U \subseteq
\prod_{i=1}^c P_i$. 
Let $X_u$ denote the complete
intersection in~$\CC^\numvars$ corresponding to $u \in U$ 
and $F_u$ be the system of $c$ polynomials in $\numvars$ variables that defines~$X_u$.

\begin{definition}\label{def:correspondence}
For $k\geq 2$, the \emph{k-bottleneck correspondence} for degree pattern $(d_1,\dots,d_c)$,
denoted $S_k$, is the set of points $(u, x_1, \dots, x_k, t_1,\dots,t_k) \in U\times \left( \CC^{\numvars}\right)^k \times \CC^k$ such that $$(x_1,\dots,x_k,t_1,\dots,t_k) \in B_k(F_u).$$
\end{definition}

We will analyze $S_k$ via projections onto its factors. 
In particular, let $\pi:S_k\to\left(\CC^\numvars\right)^k \times \CC^k$ and $\eta:S_k\to U$ be the projection maps. For any $u\in U$, the fiber $\eta^{-1}(u)$ is $\{u\}\times B_k(F_u)$.

The following provides a finiteness condition for algebraic $k$-bottlenecks.
Since geometric $k$-bottlenecks are algebraic $k$-bottlenecks, this immediately
implies a finiteness condition for geometric $k$-bottlenecks as well.

\begin{theorem}
Let $k\geq 2$ and $(d_1,\ldots, d_c)\in\NN^c$ such that each $d_i\neq 2$.
For general $u\in U$, the set of algebraic $k$-bottlenecks for $V(F_u)$ is finite.
In particular, for general $u \in U$, $\eta^{-1}(u) \setminus \Gamma_k$ is finite.
\label{thm:finitely_many_bottlenecks}
\end{theorem}

The proof of this theorem is provided at the end of this section
and follows from the Alexander-Hirschowitz Theorem \cite{Alexander1995Interp}, which is a result for homogeneous hypersurfaces. 
Thus, we need to move from affine space to projective space.
In particular, $\PP^N$ parameterizes homogeneous polynomials in $\numvars+1$ variables of degree $d$
where $N = \binom{d+\numvars}{d}-1$.  
For $a\in \PP^N$, let $F_a(x_0,\ldots,x_n)$ denote the corresponding homogeneous polynomial of degree $d$.

Let $p_1,\ldots,p_k\in \PP^n$ be general.
The Alexander-Hirschowitz Theorem considers the dimension
of the interpolation space of polynomials of degree $d$ having at least a double point 
at~$p_i$, namely
\[I_{n,k}=\left\{ a\in \PP^{N}~\left|~ F_a(p_i)=\frac{\partial F_a}{\partial x_j}(p_i)=0
\text{ for } j=1,\dots,n \text{ and }i=1,\dots, k\right\}\right..\]

\begin{theorem}[Alexander-Hirschowitz \cite{Alexander1995Interp}]
\label{thm:AH}The interpolation space $I_{n,k}$ has the {\it expected dimension}, 
i.e., $\dim(I_{n,k})=\min\{(n+1)k, N\}$, except for the following cases
\begin{itemize}
    \item $d=2, 2\leq k\leq n$;
    \item $n=2, d=4, k=5$;
    \item $n=3,d=4,k=9$;
    \item $n=4,d=3,k=7$;
    \item $n=4,d=4,k=14$.
\end{itemize}
\end{theorem}

An equivalent statement of this theorem is that
the $k$-secant variety of the $d^{\rm th}$ Veronese embedding of $\PP^n$, 
which we will call the $(n,d)$-Veronese variety $V_{n,d}$, 
has the expected dimension except for the listed exceptions.

\begin{remark}
Suppose that $p_1,\dots, p_k, q_1,\dots, q_k \in \PP^n$ 
where $k \leq n+1$ are such that $p_1,\dots,p_k$ 
and $q_1,\dots,q_k$ each span a $(k-1)$-dimensional space.
Let $I_p$ and $I_q$ denote the interpolation space $I_{n,k} \subseteq\PP^N$ as defined above, respectively.
Then, $I_p$ and $I_q$ have the same dimension. To see this, first note that there is a full rank linear map $L:\PP^n \rightarrow \PP^n$ such that $Lq_i=p_i$ for all $i$. More explicitly, complete $p_1,\dots,  p_k$ to a spanning set $p_1,\dots ,p_{n+1}$ of $\PP^n$ and similarly for $q_1,\dots,q_k$. Let $P=(p_1 \cdots p_{n+1})$ and $Q=(q_1 \cdots q_{n+1})$ be $(n+1) \times (n+1)$-matrices whose columns are the homogeneous coordinates of $p_1,\dots,p_{n+1}$ and $q_1,\dots,q_{n+1}$. Then $L$ is represented by $PQ^{-1}$. 

The group $\textrm{PGL}_n$ acts on the parameter space of hyper surfaces $\PP^N$ as follows: for \mbox{$T \in \textrm{PGL}_n$}, $Ta$ is given by the polynomial $F_a \circ T^{-1}$. Using this action and with $L$ as above, $LI_q=I_p$. In particular, by the chain rule $J(F_a \circ L)(q_i)=JF_a(p_i)L$ for all $i$ and $a \in \PP^N$.  
\end{remark}

\begin{proposition}
Let $p_1,\ldots,p_k\in \PP^n$ with $k \leq n+1$ and suppose that $p_1,\ldots,p_k$ span a $(k-1)$-dimensional subspace of $\PP^n$. With notation as above, the interpolation space has the expected dimension except if $d=2$ and $2\leq k\leq n$.\label{prop:aff-ind}
\end{proposition}
\begin{proof}
Assume that $d>2$ or that $k$ is not in $[2,n]$. Let $q_1,\dots,q_k \in \PP^n$ be independent points such that their interpolation space has the expected dimension (this is true for general $q_1,\dots,q_k$ by the Alexander-Hirschowitz theorem). Let $L: \PP^n \rightarrow \PP^n$ be a full rank linear map such that $Lq_i=p_i$ for all $i$ and let $I_p$ and $I_q$ be the interpolation spaces of $p_1,\dots,p_k$ and $q_1,\dots,q_k$, respectively. Since $I_p=LI_q$, they have the same dimension.
\end{proof}

\begin{proposition} If the $k$-secant variety  of the $(n,d)$-Veronese variety $V_{n,d}$ has the expected dimension for generic $p_1,\ldots, p_k \in\PP^n$, then the $k(n+1)$ linear forms in $(a_0,\dots,a_N)$ which comprise the entries of $JF_a(p_1), \ldots, JF_a(p_k)$ are independent. \label{prop:ind-jacob}
\end{proposition}
\begin{proof}
Let $\nu(p)=\nu_{d,n}(p)=(p^j)_{\{j\}}\in V_{n,d}$.
Then, the projective tangent space to $V_{n,d}$ 
at $\nu(p)$ is spanned by 
$\nu_{x_0}(p), \ldots, \nu_{x_n}(p)$ where $\nu_{x_i} = \frac{\partial \nu}{\partial x_i}$. The coefficients of the linear form $\frac{\partial F_a}{\partial x_i}(p)$ are the elements of the $(N+1)$-vector $\nu_{x_i}(p)$.
\end{proof}

If $c=n$, the complete intersection $X_u$ itself is
finite and \Cref{thm:finitely_many_bottlenecks} is immediate, so assume that $c<n$. We may reduce to the case where
none of the equations defining $X_u$ are linear, that is $d_i > 2$ for all $i$. Indeed, for a generic hyperplane $H \subseteq\CC^n$ there is a linear map which preserves algebraic $k$-bottlenecks of $X_u$ while eliminating a variable. After repeatedly removing linear equations we may assume that $d_i > 2$ for all $i$.

In order to show that a general complete intersection has a finite
number of algebraic $k$-bottlenecks, we need to show that the generic fiber of the projection $\eta:S_k \to U$ is finite. We do this by first studying the dimension of the fibers of the projection
$\pi:S_k \to (\prod_{i=1}^k \CC^n) \times \CC^k$.

\begin{lemma}
Let $(x_1,\dots,x_k,t_1,\dots,t_k)$ be an element in the image $\pi(S_k) \setminus \Gamma_k$. Then, $\pi^{-1}(x_1,\dots,x_k,t_1,\dots,t_k)$ has codimension $kn$ in $U$. \label{lem:codim} 
\end{lemma}
\begin{proof}
By assumption, the equations 
\[ \sum_{i=1}^k t_i = 1,\ \  z = \sum_{i=1}^k t_ix_i,\ \  d(x_1, z)^2 = d(x_j, z)^2, \ \ 2\leq j \leq k \]
are satisfied. The fiber $\pi^{-1}(x_1,\dots,x_k,t_1,\dots,t_k)$ is the algebraic subset of $U$ defined by the conditions
\[x_i \in X_u \ \text{ and }\  \text{rank}[x_i-z \ JF_u(x_i)^T]\leq c \  \text{ for } 1\leq i\leq k\text{.} \]
First note that $x_i-z = \sum_{j=1,j\not= i}^k t_j(x_i - x_j)$ for all $i$. In particular, $x_i-z$ is not $0$ because $\{x_i - x_j \}_{j=1,j\not= i}^k$ is linearly independent by assumption and none of $t_1,\dots,t_k$ is $0$. 
For all $i$, there subsequently exists a full rank $\numvars\times\numvars$ matrix $M_i$ such that $M_i(x_i-z)$ is $e_1$, the standard basis vector for $\CC^\numvars$. The fiber is equivalently defined by the conditions
\[
F_u(x_i) = 0 \ \text{ and }\ \text{rank}\begin{bmatrix}
1 & 0 & \dots & 0 \\ 
0 &  &  &  \\ 
\vdots &  & (M_i(JF_u(x_i)^T))' \\
0 & &  
\end{bmatrix} \leq c \ \text{ for } 1 \leq i \leq k
\]
where for any matrix $M$, $M'$ denotes $M$ with the first row deleted. 

We claim that the collection of forms in $u$ comprising the entries of $F_u(x_i)$ and $JF_u(x_i)$ across all $i$, $1\leq i\leq k$, is independent. Forms arising from different components of $F_u$ involve disjoint subsets of the coefficients in $u$, so it suffices to consider the case where $F_u$ is a single polynomial $f$ of degree $d > 2$. Let $\overline{f}$ denote the homogenization of $f$ and $\overline{x_i}$ denote the point in $\PP^\numvars$ with projective coordinates $[x_i;1]$. Note that since the vectors $x_1,\dots,x_k$ are affinely independent, the points $\overline{x_1},\dots,\overline{x_k}$ span a $(k-1)$-dimensional subspace of $\PP^n$ as in the statement of \Cref{prop:aff-ind}. Suppose to the contrary that a relation of the form $\sum_{i=1}^k \alpha_i f(x_i) = \sum_{1\leq i \leq k,1\leq j\leq\numvars} \beta_{ij}\frac{\partial f}{\partial y_j}(x_i)$ holds. Then the same relation holds substituting $\overline{f}$ for $f$ and $\overline{x_i}$ for $x_i$. By Euler's formula, $d\overline{f}(\overline{x_i}) = \sum_{j=1}^{\numvars+1} (\overline{x_i})_j\frac{\partial\overline{f}}{\partial y_{j}}(\overline{x_i})$. So we obtain a relation which contradicts \Cref{prop:aff-ind,prop:ind-jacob}.

We see that $\pi^{-1}(x_1,\dots,x_k,t_1,\dots,t_k)$ is a proper intersection of $k$ determinantal varieties which, by standard results, have codimension $n-c$ and a linear space defined by the linear forms $F_u(x_i)$ for $1 \leq i \leq k$. Altogether, the codimension of the fiber is $k(n-c+c) = kn$. 
\end{proof}

\begin{lemma}
 The dimension of $S_k \setminus \pi^{-1}(\Gamma_k)$ is the dimension of $U$. \label{lem:dimlem}
\end{lemma}
\begin{proof}
Consider the image $V$ of $\pi:(S_k\setminus\pi^{-1}(\Gamma_k))\to (\prod_{i=1}^k \CC^\numvars)\times \CC^k$. One can easily see 
that the image is the open algebraic subset comprised of all $(x_1,\dots,x_k,t_1,\dots,t_k)$ where
\[ \sum_{i=1}^k t_i = 1,\ \  z = \sum_{i=1}^k t_ix_i,\ \  d(x_1, z)^2 = d(x_j, z)^2, \ \ 2\leq j \leq k, \]
the $x_1,\dots,x_k$ are affinely independent, and none of the $t_i$ are 0. We claim that the image has codimension $k$, i.e., dimension $kn$. In fact, the image $V$ is birationally equivalent to~$\prod_{i=1}^k \CC^n$. The forward morphism is given by the projection map $g:V\to\prod_{i=1}^k \CC^n$ where $g(x_1,\dots,x_k,t_1,\dots,t_k)=(x_1,\dots,x_k)$. Setting $X$ to be the $\numvars\times k$ matrix whose columns are the the vectors with coordinates $x_1,\dots,x_k$, the inverse $h:\prod_{i=1}^k \CC^n \to V$ is given by taking $h(x_1,\dots,x_k)$ to be $(x_1,\dots,x_k,t_1(X),t_2(X),\dots,t_k(X))$ where the functions $t_i(X)$ are rational functions yielding the barycentric coordinates of the circumcenter of the simplex whose vertices are the columns of $X$ (see, e.g., \cite[Thm.~2.1.1]{FiedlerMatGeom} and \cite[pp.~707--708]{VanderZeeWellCentered2013}).

By \Cref{lem:codim}, the fiber of $\pi$ has codimension $kn$ for $z$ in the image $V$. It follows that $S_k\setminus\pi^{-1}(\Gamma_k)$ has dimension $\dim(U) - kn + (kn + k - k) = \dim(U)$. 
\end{proof}

Building on these results, we now present the proof
of \Cref{thm:finitely_many_bottlenecks}.

\begin{proof}[Proof of \Cref{thm:finitely_many_bottlenecks}]
Let $S$ be an irreducible component of $S_k \setminus \pi^{-1}(\Gamma_k)$. By \Cref{lem:dimlem}, $\dim{S} \leq \dim{U}$. If $\eta_{\vert S}$ does not dominate $U$, then $S \cap \eta^{-1}(u)$ is empty for general~\mbox{$u \in U$}. 
However, if $\eta_{|S}$ dominates $U$, then $\dim{S}=\dim{U}$ and $\eta^{-1}(u)$ is finite for general~\mbox{$u \in U$}.
\end{proof}

\section{Computational experiments for feature sizes}
\label{sec:examples}
This Section contains results from computing the reach, bottlenecks, and weak feature size of examples of co-dimension $1$ in $\mr^2$ and co-dimensions $1$ and $2$ in $\mr^3$. Data results and code for reproducing these computations are available at \url{https://github.com/P-Edwards/wfs-and-reach-examples}.

\begin{example}\label{ex:ButterflyAll}
Consider the ``butterfly curve'' in $\mr^2$, which is the real part of the algebraic variety defined as $V(F)\cap\mr^2$ where $F = x^4 - x^2y^2 + y^4 - 4x^2 - 2y^2 - x - 4y +1$. This example has been considered before, e.g., 
by Brandt and Weinstein \cite{brandt2019voronoi}. 

The algebraic medial axis for the butterfly curve was computed using numerical algebraic geometry 
and found to be irreducible of degree $120$.
The real part of this curve is shown in \Cref{fig:ButterflyMedial}(a) with the
pieces in cyan forming the geometric medial axis.
A lower bound of $0.103$ on the reach was estimated with a homotopy continuation method based on \Cref{cor:reach_and_lfs_hom} in agreement
with \cite[Ex.~6.1]{brandt2019voronoi}.
The points computed via \Cref{cor:reach_and_lfs_hom}
on the geometric medial axis are shown in red in~\Cref{fig:ButterflyMedial}(b).

\begin{figure}[!b]
    \centering
    \begin{tabular}{cc}
    \includegraphics[scale=0.14]{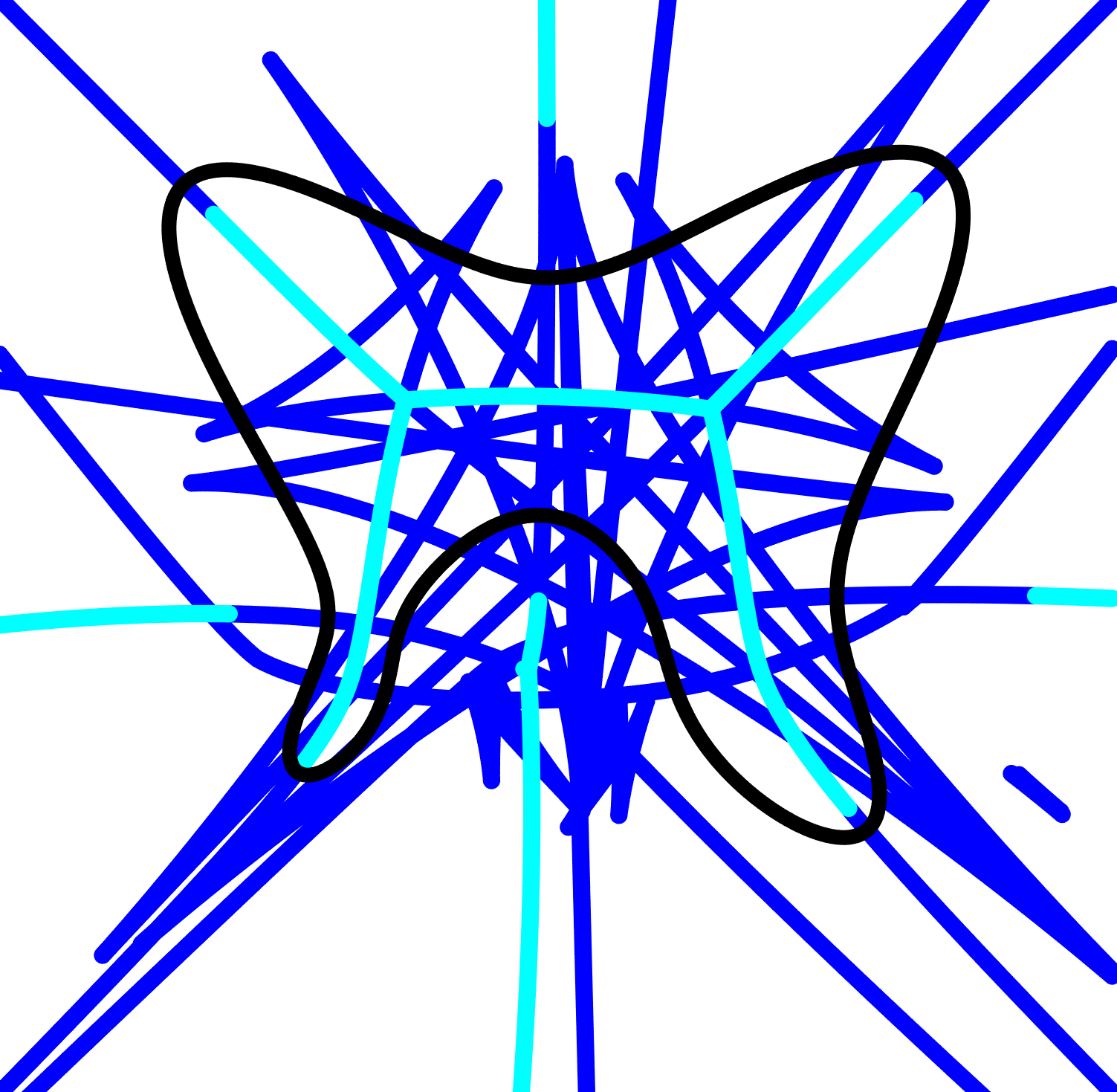} & 
    \includegraphics[scale=0.14]{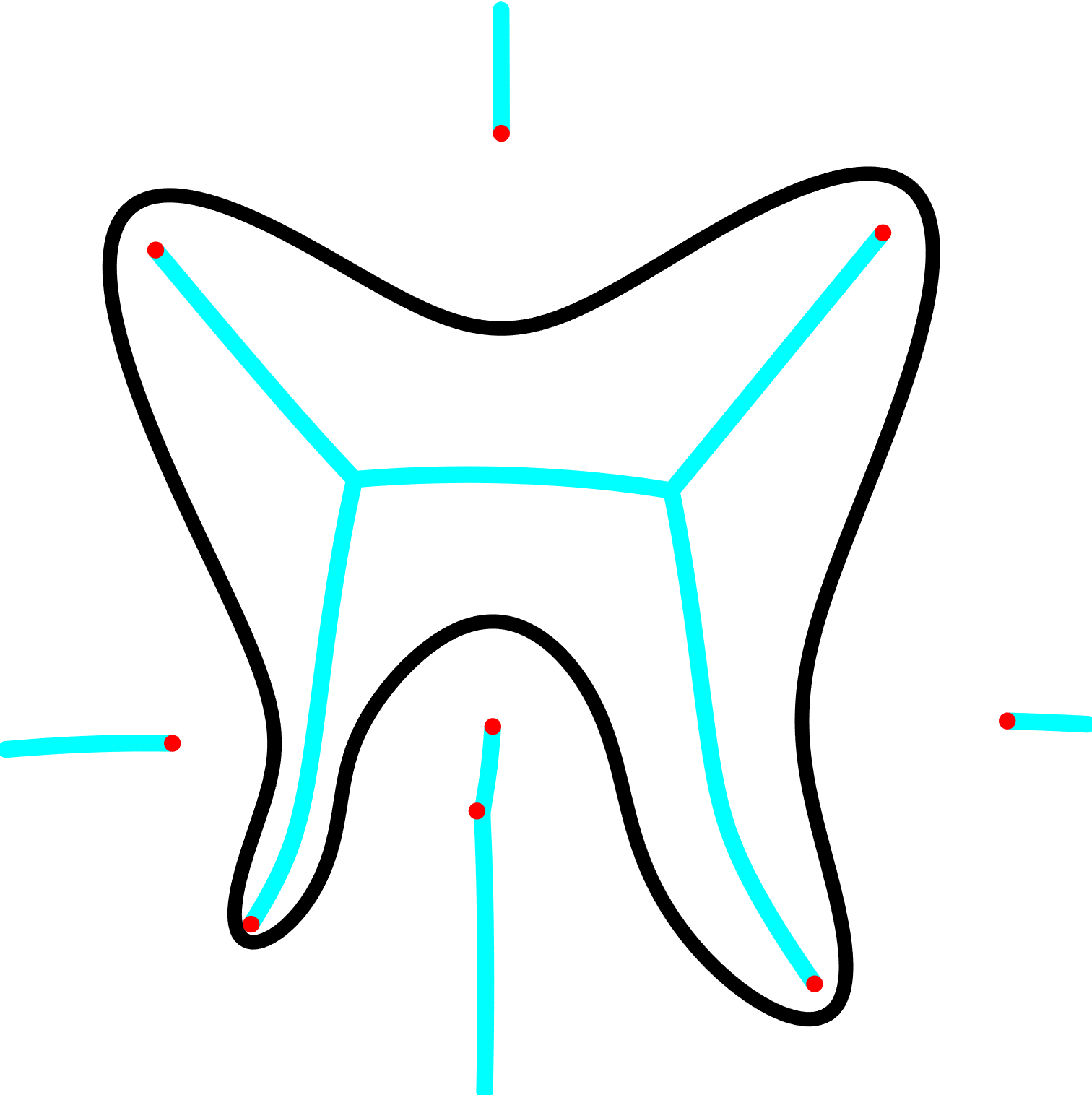} \\
    (a) & (b) 
    \end{tabular}
    \caption{For butterfly curve: (a) algebraic and geometric medial axis; (b)~geometric medial axis with critical points.}
    \label{fig:ButterflyMedial}
\end{figure}

To compute the weak feature size of $V(F)\cap\mr^2$,
we used the numerical algebraic geometric method in \Cref{cor:wfs_pd_cor}. 
For both $k=2$ and $k=3$, the results indicate that the irreducible components of $B_k(F)$ not contained in $\Gamma_k$ are all isolated points, i.e., $V(F)$ has finitely many algebraic bottlenecks. 
The following table provides a summary of the outputs. 
In particular, the weak feature size of $V(F)\cap\mr^2$ was determined to be approximately 0.251 and is attained by a geometric 2-bottleneck (cf., \cite[Ex.~6.1]{brandt2019voronoi}).
\Cref{fig:butterfly_geometric_bottlenecks,fig:butterfly_algebraic_bottlenecks} show various types of~bottlenecks
for the butterfly curve.

\begin{center}
\begin{tabular}{|c|c|c|}
    \hline 
    & $k=2$ & $k=3$  \\
    \hline
    Number of points on $B_k(G)$ computed & 392 & 2817 \\
    \hline
    Number of computed points in $\Gamma_k$ & 200 &  1089 \\
    \hline
    Algebraic $k$-bottlenecks of $V(G)$ & 96 & 288 \\ 
    \hline
    Real algebraic $k$-bottlenecks of $V(G)$ & 26 & 28 \\ 
    \hline
    Real algebraic $k$-bottlenecks of $V(G)\cap\mr^2$ & 22 & 17 \\
    \hline
    Geometric $k$-bottlenecks of $V(G)\cap\mr^2$ & 3 & 2 \\ 
    \hline 
\end{tabular} 
\end{center}

\begin{figure}
    \centering
    \includegraphics[scale=0.5]{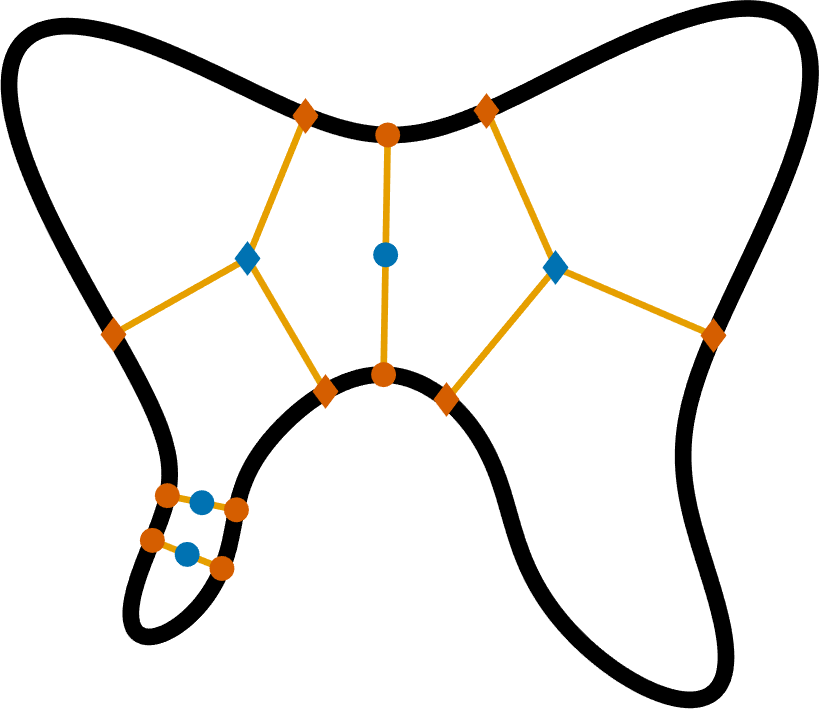}
    \caption{Geometric 2-bottlenecks (dark blue circles) and 3-bottlenecks (dark blue diamonds) of the butterfly curve. Orange dots are distance minimizers and connect to bottlenecks with light orange lines. }
    \label{fig:butterfly_geometric_bottlenecks}
\end{figure}
\begin{figure}
    \centering
    \includegraphics[scale=0.5]{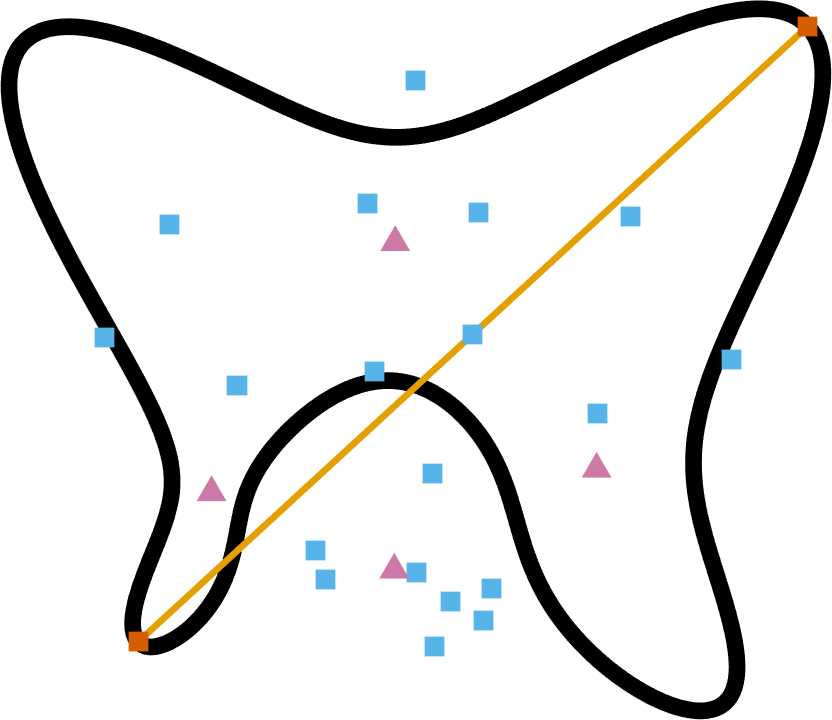}
    \hspace{0.10\textwidth}
    \includegraphics[scale=0.5]{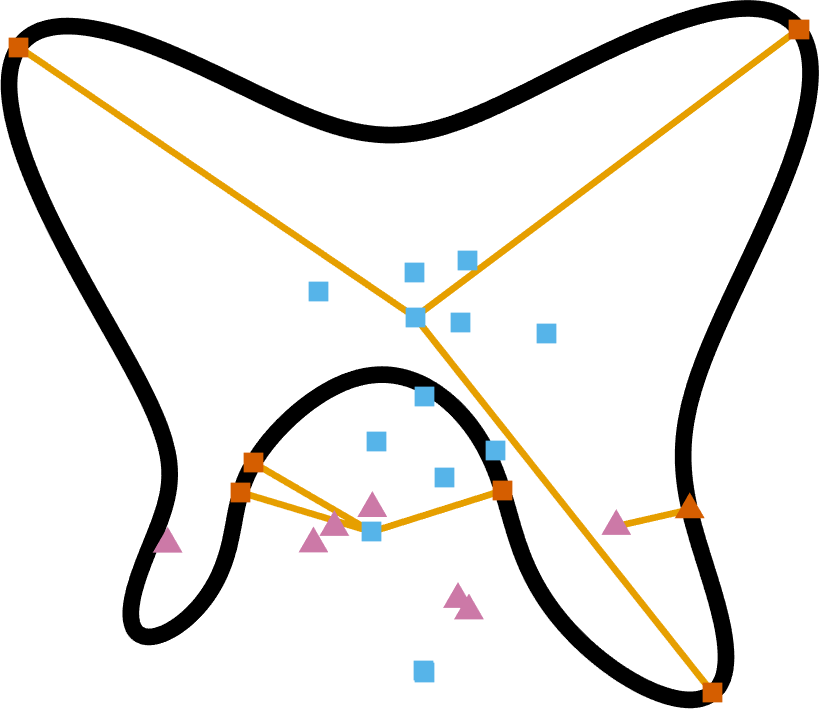}
    \caption{Algebraic 2-bottlenecks (left) and 3-bottlenecks (right) of the butterfly curve. Algebraic bottlenecks are blue and real algebraic bottlenecks of 
    the butterfly curve are pink. For an illustrative subset, orange distance-critical points are connected to their corresponding bottleneck by orange lines.}
    \label{fig:butterfly_algebraic_bottlenecks}
\end{figure}

\end{example}

\begin{example}\label{ex:clebsch}

As an example of a non-quadratic complete intersection
where \Cref{thm:finitely_many_bottlenecks} holds,
consider the intersection of a torus and Clebsch surface
in $\mr^3$ defined by 
\[
F = \left[\begin{array}{c} (R^2 - r^2 + x^2 + y^2 + z^2)^2 - 4R^2(x^2 + y^2) \\
x^3 + y^3 + z^3 + 1 - (x + y + z + 1)^3\end{array}\right]
\]
with $R= \frac{3}{2}$ and $r=1$.
In particular, the second equation 
is an algebraic surface with all 27 exceptional lines contained in $\mr^3$~\cite{Clebsch}. 
This curve is illustrated in \Cref{fig:ClebschTorus}.

\begin{figure}[!t]
    \centering
    \includegraphics[scale=0.2]{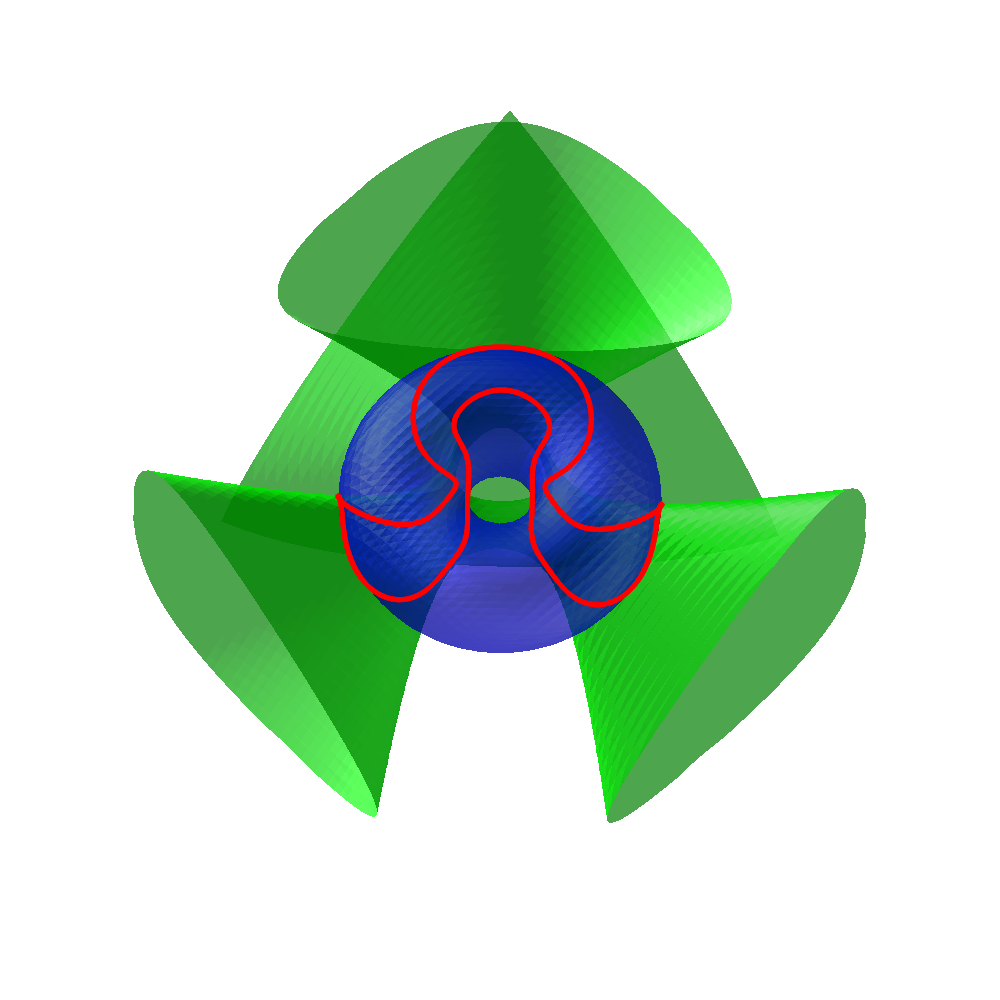}
    \caption{Curve (red) at the intersection of a torus (blue) and Clebsch surface (green).}
    \label{fig:ClebschTorus}
\end{figure}

We computed the the weak feature size for this curve
by using homotopy continuation to compute $B_k(F)$ for $k=2,3,4$. The computations indicated that the irreducible components of $B_k(G)$ not contained in $\Gamma_k$ are all isolated points.  The weak feature size is approximately 0.405, which is attained at a geometric 2-bottleneck. 
In particular, this example of computing the weak feature size is the most complicated we will consider in terms of computational cost. The cost of computing bottlenecks increases substantially with higher bottleneck order, both due to increasing the ambient dimension of $B_k(F)$ and because there are $k!$ solutions in $B_k(F)$ for each algebraic $k$-bottleneck. Regeneration methods~\cite{Regen} were used to make computations for this example more tractable. In particular, the 4-bottlenecks required approximately one week of computation on a 24-CPU computer. The table below summarizes the results.

\begin{center}
\begin{tabular}{|c|c|c|c|}
    \hline 
    & $k=2$ & $k=3$ & $k=4$ \\
    \hline
    Number of points on $B_k(F)$ computed & 2736 & 94548 & 1431936\\
    \hline
    Number of computed points in $\Gamma_k$ & 576 &  2424 & 0 \\
    \hline
    Algebraic $k$-bottlenecks of $V(F)$ & 1080 & 15354 & 59664 \\ 
    \hline
    Real algebraic $k$-bottlenecks of $V(F)$ & 68 & 324 & 586 \\ 
    \hline
    Real algebraic $k$-bottlenecks of $V(F)\cap\mr^3$ & 50 & 134 & 86 \\
    \hline
    Geometric $k$-bottlenecks of $V(F)\cap\mr^3$ & 22 & 6 & 0 \\ 
    \hline 
\end{tabular} 
\end{center}
\end{example}

\begin{example}
We conclude this collection of examples
with the quartic surface in $\mr^3$ 
from~\cite[\S5.2]{dufresne2019sampling} 
illustrated in \Cref{fig:Quartic} and defined by 
$$
F = 4x^4+7y^4+3z^4-3-8x^3+2x^2y-4x^2-8xy^2-5xy+8x-6y^3+8y^2+4y.$$ 
As in the previous examples, we computed that $V(F)$ has finitely many algebraic bottlenecks of orders $2$ and $3$. Since computing $4$-bottlenecks proved similarly expensive to \Cref{ex:clebsch}, they were not computed for this example. 

This surface exhibits interesting behavior from an algebraic viewpoint. A point $p$ approximated by $(0.458,-0.97,0)$ was computed to be the only geometric 2-bottleneck of $V(F)\cap\mr^3$ as shown in \Cref{fig:Quartic}.
The two corresponding points in $\rho_2^{-1}(p) \subseteq B_2(F)$ are isolated in the bottleneck correspondence
but are singular, i.e., have multiplicity higher than 1. 
The weak feature size of approximately $0.354$ is attained at $p$ with \Cref{fig:Quartic} also showing
the two geometric 3-bottlenecks.  The following table summarizes this computation.
\begin{center}
\begin{tabular}{|c|c|c|}
    \hline 
    & $k=2$ & $k=3$  \\
    \hline
    Number of points on $B_k(F)$ computed & 2220 & 40672 \\
    \hline
    Number of computed points in $\Gamma_k$ & 0 &  8191 \\
    \hline
    Geometric $k$-bottlenecks of $V(F)\cap\mr^3$ & 1 & 2 \\ 
    \hline 
\end{tabular} 
\end{center}

\begin{figure}
    \centering
    \includegraphics[scale=0.32]{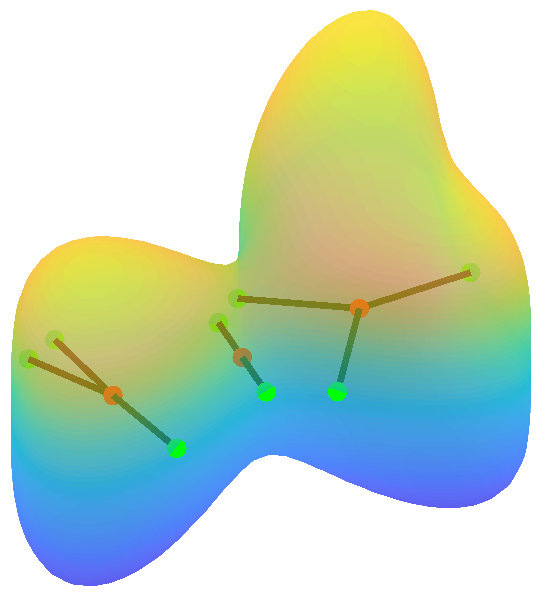}
    \caption{Quartic surface with geometric 2- and 3-bottlenecks}
    \label{fig:Quartic}
\end{figure}
\end{example}

\section{Persistent homology and adaptive sparsification} \label{sec:sparse}

Memory requirements currently comprise the most serious practical limitation for computing persistent homology \cite{roadmap}. If $\hat{X}\subseteq\mr^\numvars$ is a finite point sample, the memory required to compute the persistence diagram for the persistence module $H_\ell R_{\hat{X}}$ rapidly increases with the number of points in $\hat{X}$. One strategy to mitigate this cost is to add a ``subsampling'' step before computing persistent homology. The aim is to remove points from a sample~$\hat{X}$ using a procedure that does not substantially degrade output persistence diagrams. We can now compute both the weak feature size and local feature size of algebraic manifolds, which enables us to test algorithms on samples which fulfill theoretical density requirements. This section provides a proof-of-concept for how 
one can use our computational methods to test the behavior of geometric algorithms. 

We will consider three subsampling procedures: a uniform subsampling approach and two ``adaptive'' approaches. The latter two are based on results of Dey et al.~\cite{ddw:subsampling} and Chazal and Lieuter \cite{cl:reconstruction}, and remove points from a sample of a space $X$ based on the local feature size of $X$. More points are retained in regions where the local feature size is lower, which is a proxy for retaining more points in regions of higher curvature as illustrated in~\Cref{fig:adaptive}.

\begin{figure}
    \centering
    \includegraphics[scale=0.4]{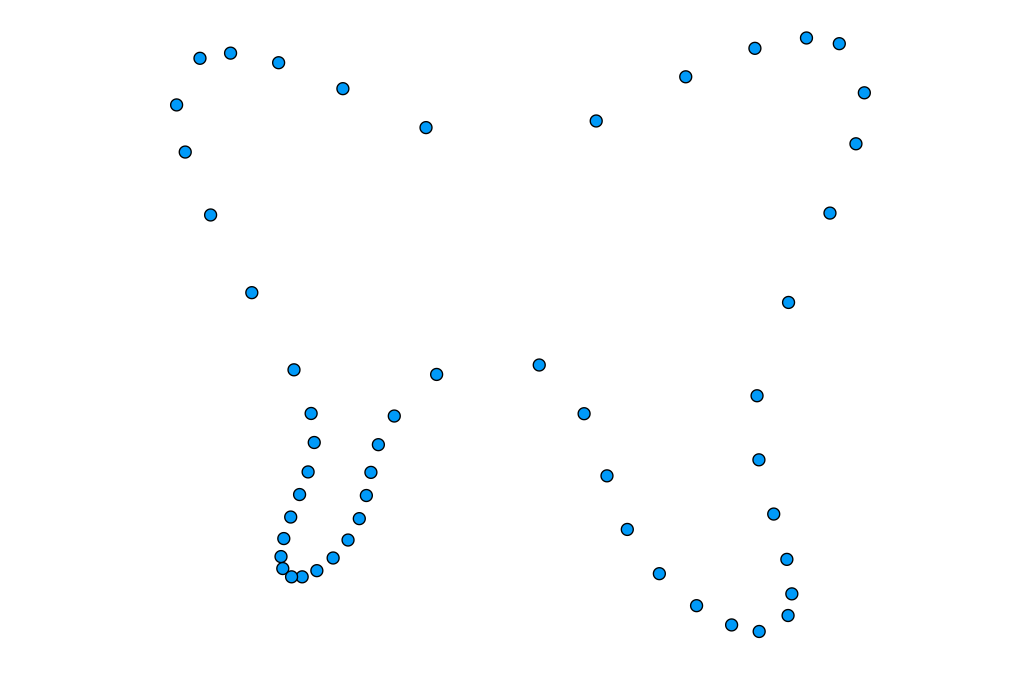}
    \includegraphics[scale=0.4]{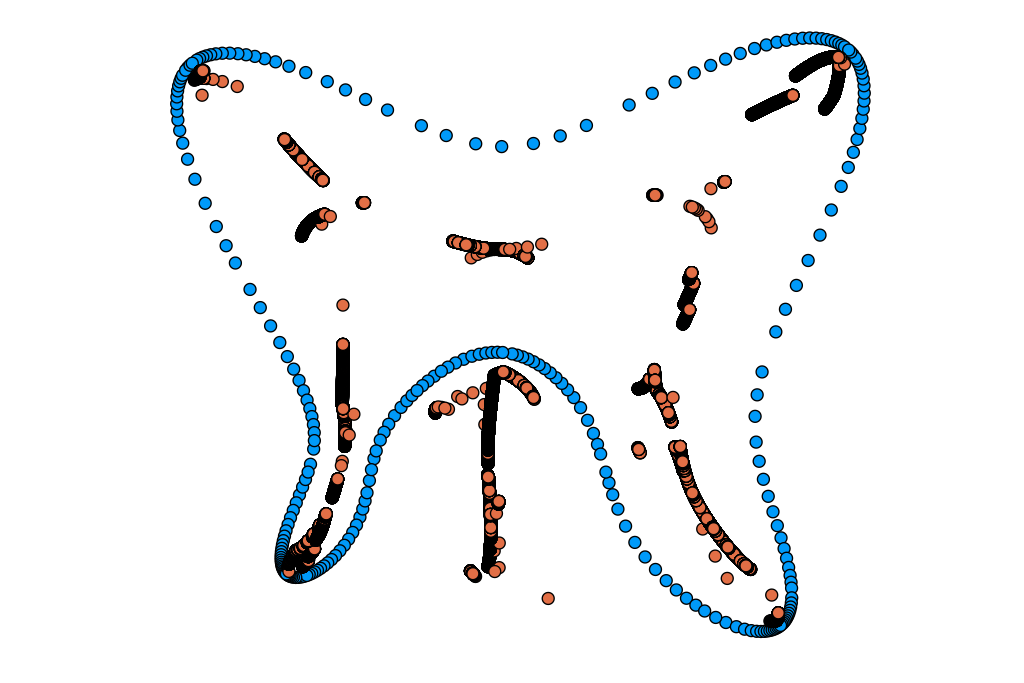}
    \caption{Adaptive samples (blue) of the butterfly curve from \Cref{ex:ButterflyAll} with respect to $\lfs$ (left) and ``lean feature size'' (right). More points are retained closer to the medial axis and ``lean medial axis'' (orange) respectively.}
    \label{fig:adaptive}
\end{figure}

\subsection{Subsampling with functions}
All three subsampling approaches we will consider fit into a greedy framework. \Cref{alg:subsampalgorithm} summarizes this framework and follows the presentation of~\cite[Alg.~1]{ddw:subsampling}.
\IncMargin{.25em}
\begin{algorithm}
    \scriptsize
	\SetKwInOut{input}{Input}\SetKwInOut{output}{Output}\SetKwFunction{Return}{Return}
	
	\input{Sample $\hat{X}\subseteq\mr^\numvars$ and function $s:\hat{X}\to\mr$}
	\output{A subsample of $\hat{X}$}
    Put $\hat{X}$ in a max priority queue sorted by $s$\;
    Set \textsc{Output} to $\emptyset$\;
    \While{The queue is not empty}{ 
        Add the highest priority $\hat{x}_1$ in the queue to \textsc{Output} and remove it from the queue\;
        Delete any point $\hat{x}_2$ from the queue where $\Vert \hat{x}_1 - \hat{x}_2 \Vert \leq s(\hat{x}_1)$\;
    }
	\Return{\textsc{Output}}
	\caption{\textsc{Subsample}}
	\label{alg:subsampalgorithm}
\end{algorithm}
\DecMargin{.25em}

\begin{example}[Uniform subsample]
Let $X = V(F)\cap\mr^\numvars$ where $V(F)$ is smooth and equidimensional. We can compute $\omega < \wfs(X)$ via homotopy continuation as we did in \Cref{sec:examples}. For any $\lambda\in [0,1]$, define $\omega_\lambda:\mr^\numvars\to\mr$ to be the constant function given by $\omega_\lambda(z) = \lambda \omega$. The output of $\textsc{Subsample}(\hat{X},\omega_\lambda)$ is a ``uniform subsample'' of $\hat{X}$. If $\hat{X}$ is a $(\delta,\mu \omega)$-sample of $X$, then $\textsc{Subsample}(\hat{X},\omega_\lambda)$ is a $(\delta,(\mu+\lambda)\omega)$-sample of $X$.
\label{ex:uniform}
\end{example}

\begin{example}[Local adaptive subsample]
With $X$ and $F$ as in the previous example and 
$\lfs$ being the local feature size function of $X$, for any $\lambda\in [0,1]$, define $\lfs_\lambda:\mr^n\to\mr_{\geq 0}$ by $\lfs_\lambda(z) = \lambda\lfs(z)$. The output of $\textsc{Subsample}(\hat{X},\lfs_\lambda)$ is an ``adaptive subsample'' of $\hat{X}$ with respect to the local feature size. In practice, we compute $\widehat{\lfs}(z) \leq \lfs(z)$ for any $z\in\hat{X}$ via homotopy continuation as in \Cref{cor:reach_and_lfs_hom}. 
\label{ex:lfs_sparse}
\end{example}

\begin{example}[Lean adaptive subsample]
Let $\hat{X}\subseteq\mr^\numvars$ be a point sample. In \cite[Def. 3]{ddw:subsampling}, Dey et al. define the $\frac{\pi}{5}$\emph{-lean set} of $\hat{X}$, $L_\frac{\pi}{5}$, to be a subset of the set of midpoints $\{\frac{p+q}{2} \}_{p\not=q \in \hat{X}}$ that also fulfills other geometric conditions (see \cite{ddw:subsampling} for details). In particular, the $\frac{\pi}{5}$-lean set of $\hat{X}$ can be computed using just the points in $\hat{X}$ as input. 
They also show that, if $\hat{X}$ is a dense sample from a manifold in an appropriate sense, then the distance function $d_{L_{\frac{\pi}{5}}}$ estimates the distance to a subset of the medial axis of $X$. For $\lambda\in [0,1]$ define $\lnfs_\lambda:\mr^\numvars\to\mr_{\geq 0}$ by $\lnfs_\lambda(z) = \lambda d_{L_{\frac{\pi}{5}}}(z)$. The output of $\textsc{Subsample}(\hat{X},\lnfs_\lambda)$ is an adaptive subsample of $\hat{X}$ with respect to the lean feature size. 
\label{ex:lnfs_sparse}
\end{example}

\subsection{Adaptive Vietoris-Rips complexes}
Standard Vietoris-Rips persistent homology is not appropriate when $\hat{X}$ is an adaptive subsample. Instead, we must use an adaptive Vietoris-Rips complex. 

\begin{definition}
Let $\hat{X}$ be a finite subset $\mr^\numvars$ and let  $r:\hat{X}\to\mr_{\geq 0}$. For any non-empty $\sigma\subseteq \hat{X}$, define $\diam_r(\sigma) = \max_{p,q\in\sigma} \frac{\Vert p - q \Vert}{r(p) + r(q)}$. Then, for $t\geq 0$, the \emph{$r$-adaptive Vietoris-Rips complex} of $\hat{X}$ at threshold $t$ is the abstract simplicial complex
\[
R_{\hat{X},r}(t) := \{ \sigma\not=\emptyset \subseteq \hat{X} \mid \diam_r(\sigma) \leq t \}\text{.}
\]
\end{definition}

When $r$ is the constant function with output $1/2$, $R_{\hat{X},r}$ is the standard Vietoris-Rips complex. We will use $r=\lfs$ when $\hat{X}$ is an adaptive subsample with respect to $\lfs$ and similarly for $\lnfs$. Note that computing the persistence diagram for $H_\ell R_{\hat{X},r}$ is straightforward with standard software, as we can provide input in the form of a matrix for the function $M:\hat{X}\times\hat{X}\to\mr$ given by $M(p,q) = \frac{\Vert p - q\Vert}{r(p) + r(q)}$.

We can now consider a persistent homology pipeline with subsampling that has two parameters, $\mu$ and $\lambda$. Assume that $\delta > 0$ is fixed.

\begin{algorithm}[!h]
    \scriptsize
	\SetKwInOut{input}{Input}\SetKwInOut{output}{Output}\SetKwFunction{Return}{Return}
	
	\input{A polynomial system $F$ defining a smooth and equidimensional algebraic variety with $X = V(F)\cap\mr^n$ compact and a degree of homology $\ell \geq 0$}
	\input{A sampling density $\mu > 0$ and subsampling proportion $\lambda \geq 0$}
	\input{A family $\mathcal{S}$ of subsampling functions with a function $s_\lambda:\mr^\numvars\to\mr_{\geq 0}$ for each $\lambda\in\mr_{\geq 0}$ and a radius function $r:\mr^\numvars\to\mr$ }
	\output{A persistence diagram}
    Compute $0 < T < \wfs(X)$ or $0 < T < \reach(X)$\;
    Compute using $T$ a $(\delta,\mu T)$-sample of $X$, $\hat{X}_0$ \;
    Set $\hat{X}$ to \textsc{Subsample}$(\hat{X}_0,s_\lambda)$ \;
    Compute and return the Vietoris-Rips persistence diagram of $H_\ell R_{\hat{X},r}$
	\caption{\textsc{PH with Subsampling}}
	\label{alg:subph}
\end{algorithm}
\DecMargin{.25em}

To justify using adaptive complexes for persistence computations, we would like analogs to the Homology Inference Theorem. For $r=\lnfs$, this is provided in~\mbox{\cite[Thm. 3.1]{ddw:subsampling}}. For $r=\lfs$, 
this follows essentially from a result of Chazal and Lieutier~\cite[Thm. 6.2]{cl:reconstruction}, but requires additional modifications. 
Since proving these modifications work is technical and the arguments are mostly standard, this is left to \Cref{asec:inference}.

\subsection{Computational results from the butterfly curve}
The following table summarizes valid parameters for homology inference with \Cref{alg:subph} computed for the butterfly curve from \Cref{ex:ButterflyAll}
based on applicable homology inference theorems. The sampling density condition for homology inference in \cite{ddw:subsampling} is difficult to compute as its relationship to the reach or weak feature size is not obvious. The sampling density $\mu$ for $\lnfs$ is therefore the same as for $\widehat{\lfs}$ and the entries otherwise follow those in~\cite{ddw:subsampling}.
\begin{table}[H]
\begin{center}
\begin{tabular}{|c|c|c|c|c|c|c|}
    \hline 
    $\mathcal{S}$ & r & $\mu$ & $\lambda$ & T &  a & b  \\
    \hline
    $w_\lambda$, \Cref{ex:uniform} & 1/2 & 0.15 & 0.15 &  $\wfs$ & 0.0753 & 0.348 \\
    \hline 
    $\widehat{\lfs}_\lambda$, \Cref{ex:lfs_sparse} & $\widehat{\lfs}$ & 0.0046 & 0.0019 & $\reach$ & 0.0111 & 0.023 \\
    \hline 
    $\lnfs_\lambda$, \Cref{ex:lnfs_sparse} & $\lnfs$ & 0.0046 & 0.009 & $\reach$ & 0.018 & 0.108 \\
    \hline
\end{tabular} 
\caption{Input parameters for \Cref{alg:subph} which compute a persistence diagram where points above and to the left of the indicated $(a,b)$ count the $\ell^\text{th}$ Betti number of the butterfly curve.}
\label{tab:parameters}
\end{center}
\end{table}

To compare these three methods it is natural to first compute a sample $\hat{X}$ as in \Cref{alg:subph} with the parameter values in the above Table, compute samples and persistence diagrams for $\textsc{Subsample}(\hat{X},s_\lambda)$ while varying the subsampling parameter $\lambda$ between $0$ and~$1$, and compare outputs. When $\lambda = 0$, the samples fulfill homology inference conditions and the persistence diagrams degrade as $\lambda$ increases to $1$ since more points are removed.

To be definite, for any fixed row in \Cref{tab:parameters}, let $\hat{X}$ be a sample of the butterfly curve computed by \Cref{alg:subph} with those parameters. Then, for any $\lambda\in [0,1]$, denote by $\hat{X}_\lambda$ the output of $\textsc{Subsample}(\hat{X},s_\lambda)$ and by $D_\lambda$ the persistence diagram of $H_1R_{\hat{X}_\lambda,r}$. Consider the following scores which summarize these outputs: 
\begin{enumerate}
    \item (Computational cost score) The score for $\lambda$ is the number of points $\# \hat{X}_\lambda$. Lower numbers of points are more desirable for computations.
    \item (Homology inference score) The butterfly curve has $\beta_1 = 1$. This score measures the ease of estimating that $\beta_1=1$ rather than $\beta_1 > 1$ using the persistence diagram~$D_\lambda$. For any $\lambda$, let $\text{pers}_{\lambda,1}$ be the largest value in $\{ d - b \}_{(b,d)\in D_\lambda}$ and $\text{pers}_{\lambda,2}$ be the second largest value. Set $M_\lambda = \text{pers}_{\lambda,1} - \text{pers}_{\lambda,2}$. The homology inference score for sparsification level~$\lambda$ is $\frac{M_\lambda}{M_0}$.  The score is between $0$ and $1$, with a higher score indicating a better persistence diagram for homology inference.
    \item (Wasserstein score) The $2$-Wassertein distance, denoted $W_2$, is a standard metric\footnote{More precisely in our setting, an \emph{extended} metric, which means the distance between two diagrams may be $\infty$. Readers familiar with algebraic approaches to persistent homology may find it useful to note that the persistence diagrams in this paper are restricted: they contain finitely many points and do not include points on the diagonal.} for persistence diagrams (e.g., \cite[p. 183]{HarerBook}). The Wasserstein score for $\lambda$ is $\frac{W_2(D_0,D_\lambda)}{W_2(D_0,\emptyset)}$. Lower scores correspond to higher quality persistence diagrams.
\end{enumerate}
\begin{figure}[!b]
    \centering
    \includegraphics[scale=0.34]{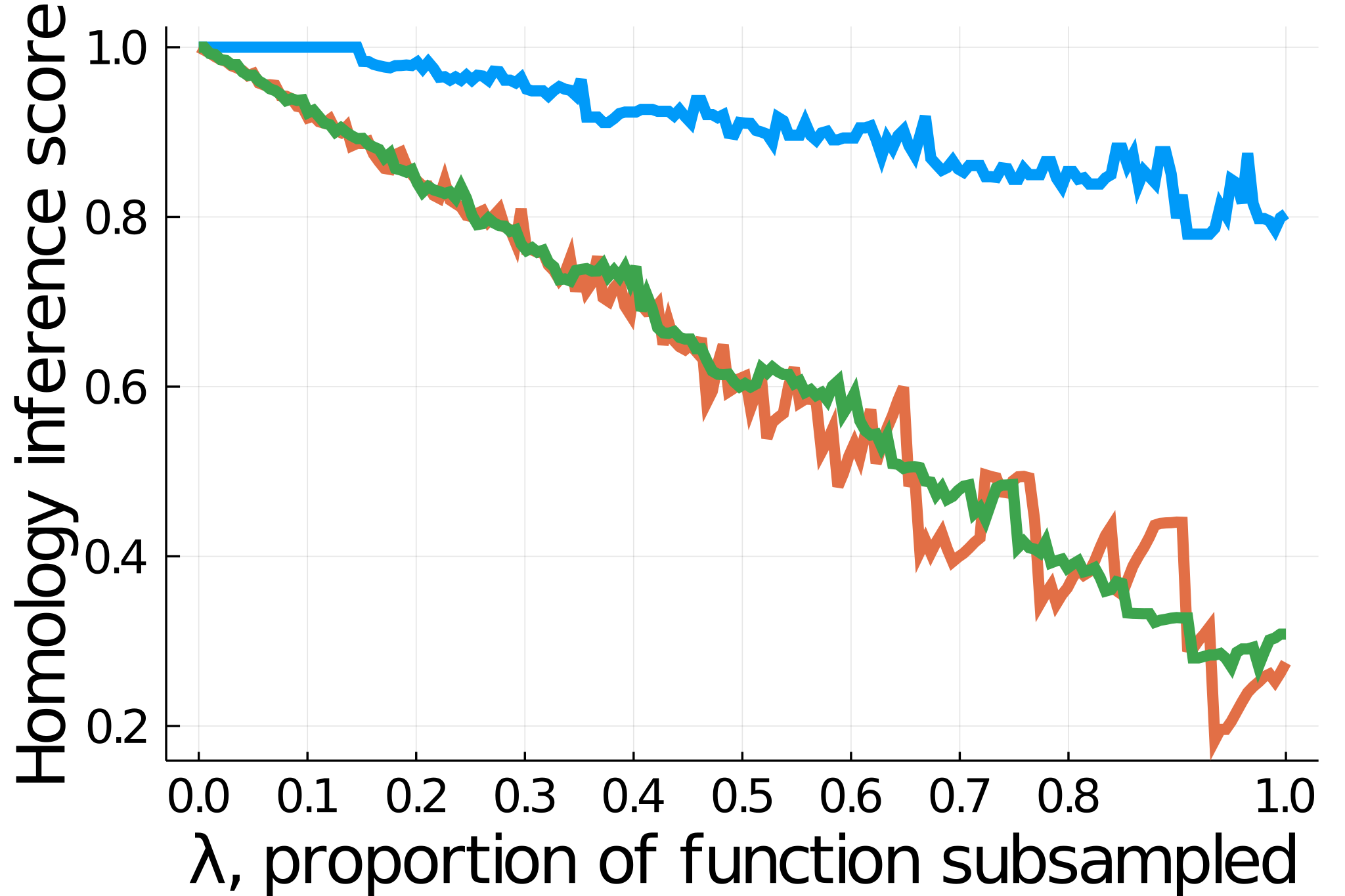}
    \includegraphics[scale=0.34]{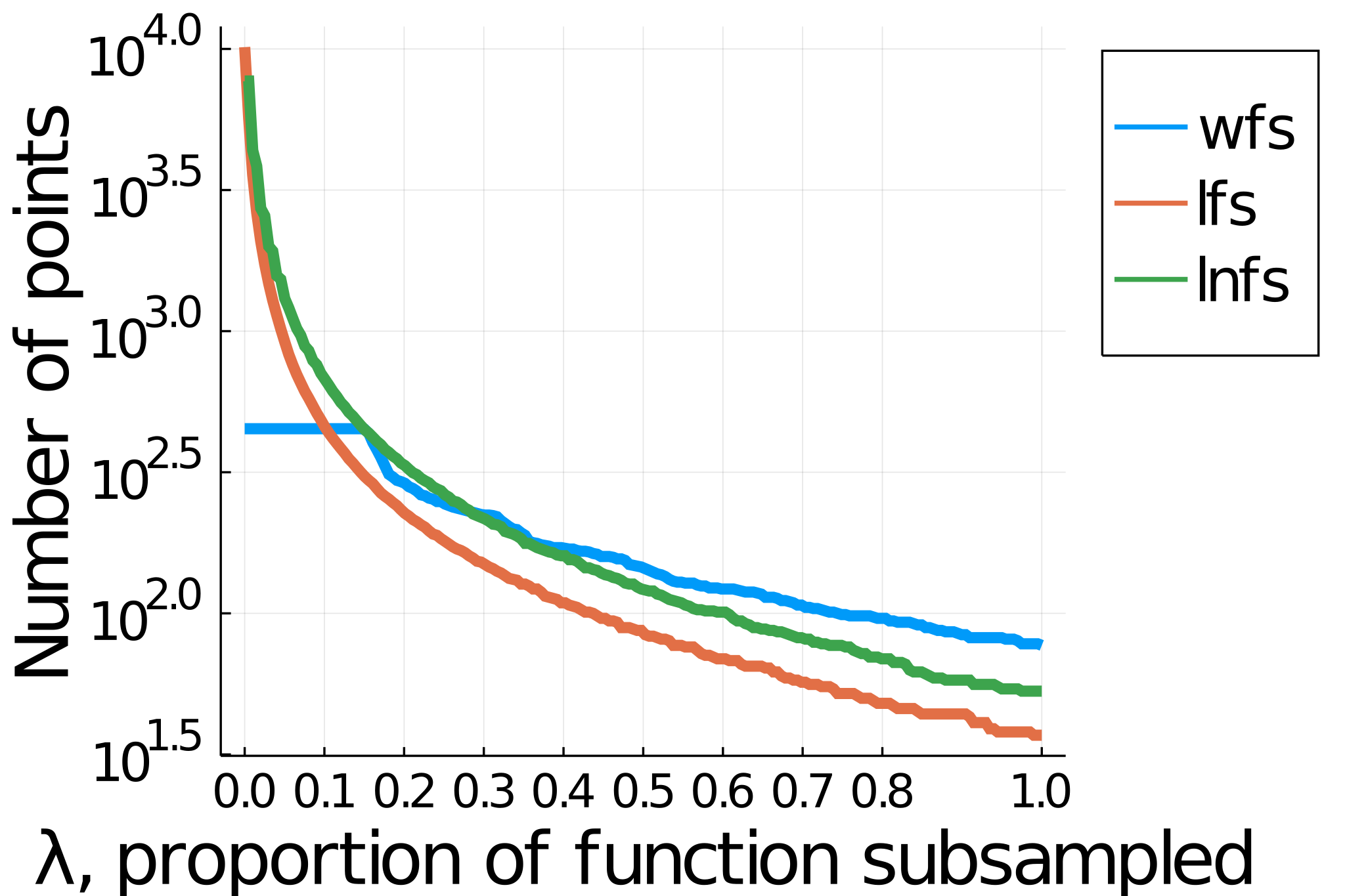}
    \includegraphics[scale=0.34]{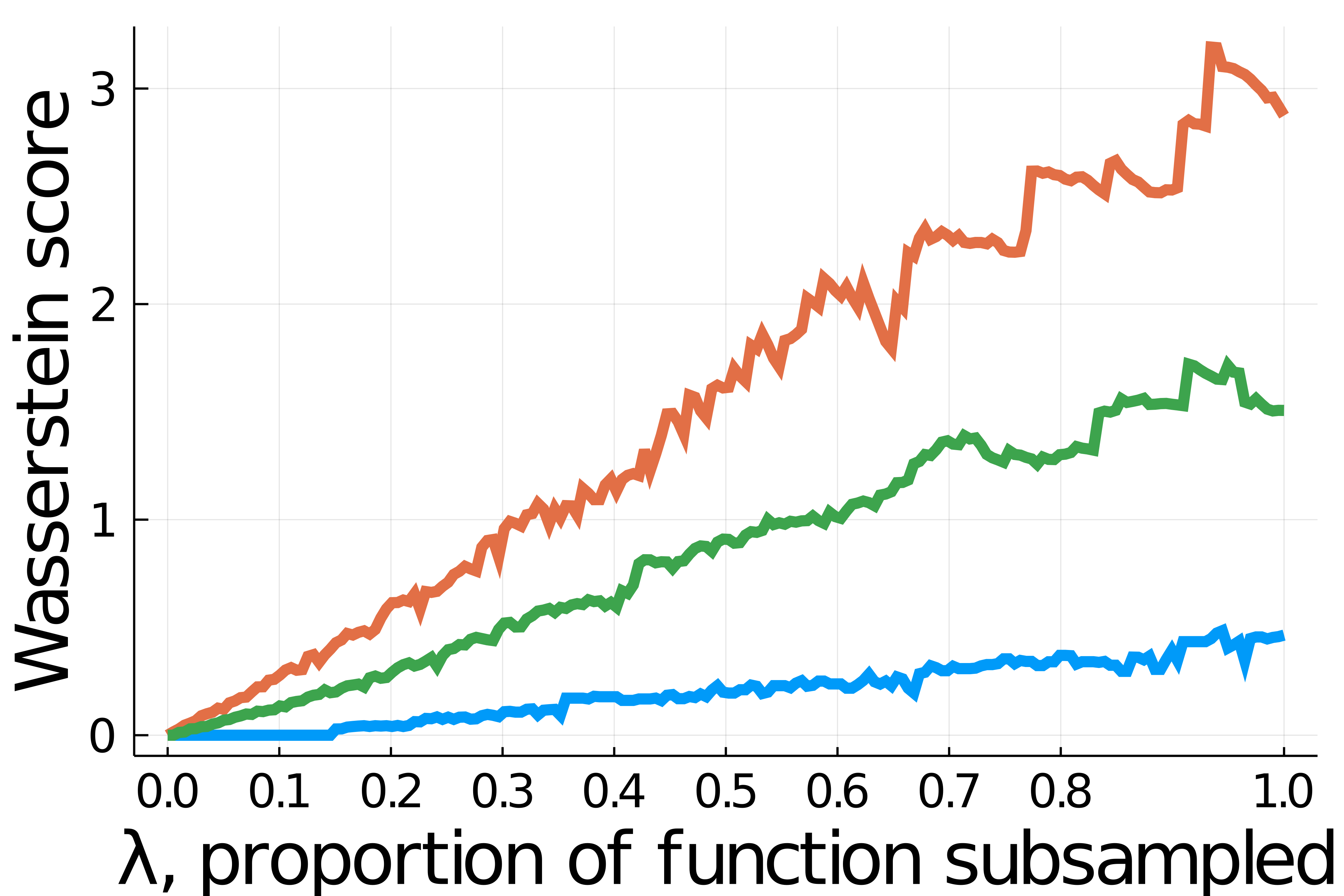}
    \caption{Computational results comparing the behavior of subsampling methods.}
    \label{fig:butterfly_sparse_numpoints}
\end{figure}

\Cref{fig:butterfly_sparse_numpoints} records results from experiments that sampled the butterfly curve and computed the scores above. In particular, 
the $\lnfs$-adaptive subsampling exhibits comparable performance to the $\lfs$-adaptive subsampling method it approximates. Both of these adaptive methods produce subsamples with fewer points than uniform subsampling across a substantial range of subsampling thresholds. We can also see that the persistent homology outputs were less sensitive to the subsampling threshold when conducting uniform subsampling.

\section{Conclusion}\label{sec:Conclusion}
In this paper, we developed theoretical foundations and numerical algebraic geometry methods for computing geometric feature sizes of algebraic manifolds. We also demonstrated how to combine these methods with persistent homology for both homology inference and for testing geometric algorithms. This study is not intended to be exhaustive, so some further questions both in terms of theory and applications follow. 

\noindent \textbf{Real algebraic spaces with singularities.} It is natural to ask how the results presented here may generalize to singular spaces. 
Since isolated singularities can contribute additional irreducible components to $B_k(F)$, 
the impact of singularities must be analyzed.

\noindent \textbf{Counting algebraic bottlenecks.} A direct consequence of \Cref{thm:finitely_many_bottlenecks}, which will be familiar to readers who have worked with parameter homotopies, is that, for a fixed degree pattern, there exist upper bounds on the number of algebraic (and so geometric) bottlenecks that apply for any generic algebraic manifold with that degree pattern. Computing sharp upper bounds, however, is an open problem of more than intrinsic interest. As an example of the geometric meaning of these bounds, consider a compact algebraic hypersurface $H\subseteq\mr^\numvars$, not necessarily smooth, e.g., the discriminant locus of a parameterized family. 
Thus, $\mr^\numvars\setminus H$ decomposes into a finite number of disconnected $\numvars$-cells and the number of geometric bottlenecks of $H$ 
is an upper bound on the number of cells. 
Altogether, having good bounds on this number both for algebraic manifolds and for singular algebraic spaces could be useful for geometric algorithms which look to estimate the number and size of these cells.

\noindent \textbf{Algebraic models and persistent homology.} In \Cref{sec:lfs_compute}, we saw an application where feature sizes are computed to construct samples of an algebraic manifold $X$ for analysis via Vietoris-Rips persistent homology. Instead, consider the persistence module obtained by thickening the space being sampled, i.e., the persistence module $H_\ell X^\bullet$. Its persistence diagram is constrained by the critical values of $d_X$, which we can now compute. That persistence diagram in turn constrains the persistence diagram obtained from Vietoris-Rips persistent homology of a sample~\cite{harker2019comparison}. Can these constraints be leveraged to reduce the cost of persistence computations from a sample?

\noindent \textbf{Reducing redundant computations.}
For any polynomial system $F$, there is an action of the symmetric group on $k$ elements on the bottleneck correspondence $B_k(F)$.  Namely, a permutation acts on an element $(x_1,\dots,x_k,t_1,\dots,t_k)$ by permuting both the $x_i$ and~$t_i$. In the generic case when all non-degenerate solutions are isolated, standard homotopy continuation methods whose results we saw in this manuscript compute $k!$ solutions in $B_k(F)$ for each algebraic bottleneck. 
Is there a natural approach, e.g., building on methods utilized in~\mbox{\cite{SymmetricTensor,SymmetricHomotopy,NinePoint}}, 
that takes advantage of the symmetry to reduce these redundancies?
 
\section*{Acknowledgements}
PBE thanks Antonio Lerario for interesting discussions, and both Parker Ladwig and Robert Goulding for help tracking down and verifying Cassini's study of his eponymous ovals.
JDH was supported in part by NSF grant CCF-181274. OG was supported in part by EPSRC EP/R018472/1 and Bristol Myers Squibb. For the  purpose of Open Access, the author has applied a CC BY public copyright licence to any Author Accepted Manuscript (AAM) version arising from this submission.


\appendix
\section{Appendix}\label{appendixA}

\subsection{Proof of \texorpdfstring{\Cref{thm:hom_inference}}{Theorem~2.11}}
\begin{theorem}
Let $A_\numvars = \sqrt{\frac{2\numvars}{\numvars+1}}$, let $\hat{X}$ be a $(\delta,\epsilon)$-sample of a real semialgebraic set $X$ in $\mr^\numvars$, and let $\delta' = 2\epsilon(A_n^2 -1)+\delta A_n$. If $2(\epsilon + \delta') < \wfs(X)$ then for $a(\delta,\epsilon) = 2\epsilon$ and $b(\delta,\epsilon) = 2(2\epsilon A_n + \delta)$ the Betti number $\beta_\ell(X)$ is the rank of the map obtained from applying $H_\ell$ to the inclusion map $R_{\hat{X}}(a(\delta,\epsilon)) \subseteq R_{\hat{X}}(b(\delta,\epsilon))$.
\end{theorem}
\begin{proof}
We have the following chain of simplicial inclusions: 
\[
C_{\hat{X}}(\epsilon) \subseteq R_{\hat{X}}(2\epsilon) \subseteq C_{\hat{X}}(\epsilon A_n) \subseteq C_{\hat{X}}(2\epsilon A_n + \delta) \subseteq R_{\hat{X}}(2(2\epsilon A_n + \delta))
\]
\[
 \subseteq C_{\hat{X}}(A_n(2\epsilon A_n + \delta)) = C(2\epsilon + \delta')
\]
where inclusions between \u{C}ech and Vietoris-Rips complexes follow from de Silva and Ghrist's Theorem \cite{coveragetop}. Note that since $\hat{X}$ is a $(\delta,\epsilon)$-sample of $X$, it is also a $(\delta',\epsilon)-$sample of $X$ and a $(\delta,\epsilon A_n)-$sample of $X$ because $\delta' \geq \delta$ and $\epsilon A_n > \epsilon$. Also note that $2(\epsilon A_n+\delta) < 2(\epsilon+\delta') < \wfs(X)$, the first inequality being easy to verify given $A_n > 1$ and the second inequality having been assumed. Applying $H_\ell$ to the sequence, the linear maps induced by both the internal and external inclusions of \u{C}ech complexes have rank $\beta_\ell(X)$ by \Cref{thm:cech_hom_inf}. By properties of linear maps, these inclusions give an upper and lower bound, respectively, of $\beta_\ell(X)$ for $\rank(H_\ell R_{\hat{X}})(a(\delta,\epsilon),b(\delta,\epsilon))$. 
\end{proof}

\subsection{Homology inference and subsampling}
\label{asec:inference}
Recall that for a set $X$ in $\mr^\numvars$ and $z\in\mr^\numvars$, $\pi_X(z)$ denotes the set of points in $X$ with minimum distance to $z$. Abusing notation, let $\lfs(\pi_X(z))$ denote $\inf_{x\in\pi_X(z)} \lfs(x)$. This is always a minimum for closed $X$ because $\lfs$ is continuous, in fact 1-Lipschitz continuous. 
\begin{definition} (\cite{cl:reconstruction})\label{def:adapt}
A finite subset $\hat{X}$ of $\mr^\numvars$ is an \emph{adaptive-$(\alpha,\beta)$ sample} of a compact manifold $X$ for $\alpha,\beta > 0$ if 
\begin{itemize}
    \item for any $\hat{x} \in \hat{X}$, $d_X(\hat{x}) < \alpha\beta\lfs(\pi_X(\hat{x}))$, and
    \item for all $x\in X$, there is $\hat{x}\in\hat{X}$ such that $d(x,\pi_X(\hat{x})) < \beta\lfs(\pi_X(\hat{x}))$.
\end{itemize}
\end{definition}
\begin{definition}(\cite{cl:reconstruction}) Suppose $0 < \kappa$, $X$ is a compact subspace of $\mr^\numvars$, and $P$ is a finite subset of $\mr^\numvars$. For any $z\in\mr^\numvars$ let $\lfs^\pi_\kappa(z) = \kappa\lfs(\pi_X(z))$. Denote by $K_{P,\lfs^\pi}(\kappa)$ the union of~balls 
\[
\cup_{z\in P} \overline{B}_z(\lfs^\pi_\kappa(z))
\]
and denote by $C_{P,\lfs^\pi}$ the functor $\RR\to\simp$ where $C_{P,\lfs^\pi}(\kappa)$ is the nerve of $\{\overline{B}_z(\lfs^\pi_\kappa(z))\}_{z\in P}$.\label{def:lfs_balls}
\end{definition}

\begin{theorem} (Chazal and Lieutier~\cite[Thm. 6.2]{cl:reconstruction}) Suppose $\hat{X}$ is an $(\alpha,\beta)-$adaptive sample of a smooth 
and compact manifold $X$ in $\mr^\numvars$. There exist functions $g,h:\mr^4\to\mr$ where, if $0 < a < b < \frac{1}{3} - \alpha\beta$ and $g(\alpha,\beta,a,b) < h(\alpha,\beta,a,b)$, then $X$ is a deformation retract of $K_{\hat{X},\lfs^\pi}(\kappa)$ for $\kappa\in [a,b]$.
\label{thm:adapt_hom_inference}
\end{theorem}

The Nerve Theorem applies in the above definition, so that the geometric realization $\vert C_{P,\lfs^\pi}(\kappa)\vert$ is homotopy equivalent to $K_{P,\lfs^\pi}(\kappa)$ for all $P$, $X$, and $\kappa$. The functions $g$ and~$h$ in the Theorem are given explicitly by Chazal and Lieutier, and arise from technical 
geometric~considerations. 

Our homotopy continuation methods in \Cref{sec:lfs_compute} give us oracles for $\lfs$ and the reach, but these require some care to integrate with \Cref{thm:adapt_hom_inference}. The oracles introduce some estimation error, for instance, and we can only estimate $\lfs(z)$ for a point $z\in\mathbb{R}^\numvars$ rather than~$\lfs(\pi_X(z))$. 

To fix notation, let $X$ be a smooth and compact algebraic manifold in $\mr^\numvars$ that is the real part of an equidimensional and smooth algebraic variety. Let $0 < \widehat{\lfs} \leq \lfs_X$, let $\Vert\lfs - \widehat{\lfs}\Vert_\infty \leq E_{\lfs}$, and let $0 < R < \reach(X)$. Fix $\delta > 0$ and denote $\frac{\delta}{R}$ by $\delta_R$. In the following, recall that $\lfs(x) \leq \reach(X)$ for all $x\in X$ and that $\lfs$ is 1-Lipschitz.

\begin{theorem}
Let $0 < \lambda,\mu \leq 1$ and let $\hat{X}_0$ be a $(\delta_R R,\mu R)$-sample of $X$. 
\begin{enumerate}
    \item Take $\beta' = \mu+3\delta_R+\lambda(1+\delta_R)$ and 
    $\alpha' = \frac{\delta_R}{\beta'}$. If $\widehat{\lfs}_\lambda:\mr^\numvars\to\mr$ is defined by $z\mapsto \lambda\widehat{\lfs}(z)$, then 
    $\textsc{Subsample}(\hat{X}_0,\widehat{\lfs}_\lambda)$, denoted $\hat{X}$, is an 
    adaptive-$(\alpha',\beta')$ sample of $X$.
    
    \item Set $M_K = \frac{E_{\lfs}}{R} + \delta_R +1$ and 
    $M_{\hat{K}} = \frac{E_{\lfs}}{\hat{R}(1-\delta_R)} + \frac{1}{1-\delta_R}$ where 
    $\hat{R} = \min_{\hat{x}\in\hat{X}} \widehat{\lfs}(\hat{x})$. If there is $a > 0$ such that $4(M_KM_{\hat{K}})^2 a < \frac{1}{3} - \delta_R$ and
    $g(\alpha',\beta',a,4(M_K M_{\hat{K}})^2a) < 
    h(\alpha',\beta',a,4(M_K M_{\hat{K}})^2a)$ where $g$ and $h$ are 
    the functions from \Cref{thm:adapt_hom_inference}, then the rank of
    the map
    \[
    H_\ell(R_{\hat{X},\widehat{\lfs}}(M_k a) \subseteq R_{\hat{X},\widehat{\lfs}}(2M_K^2 M_{\hat{K}}a))
    \]
    is the $\ell^\text{th}$ Betti number of $X$.
\end{enumerate}
\end{theorem}
The remainder of the Appendix is dedicated to proving this Theorem.
\begin{proposition}
Let $0 \leq \lambda \leq 1$, let $X$ have $0 < R < \reach(X)$, let $\hat{X}$ be a $(\delta_R R,\mu R)$-sample of $X$ in $\mr^\numvars$ and let $\widehat{\lfs}_\lambda:\mr^\numvars\to\mr$ be defined by $z\mapsto \lambda\widehat{\lfs}(z)$. Take $\beta' = \mu + 3\delta_R + \lambda(1+\delta_R)$ and $\alpha' =\frac{\delta_R}{\beta'}$. Then $\textsc{Subsample}(\hat{X},\widehat{\lfs}_\lambda)$ is an adaptive-$(\alpha',\beta')$ sample of $X$.
\end{proposition}
\begin{proof} 
Denote $\textsc{Subsample}(\hat{X},\widehat{\lfs}_\lambda)$ by $S$. The first condition of \Cref{def:adapt} is trivially satisfied because $S \subseteq \hat{X}$. To see the second condition holds, first suppose $x\in X$. There is $\hat{x}\in\hat{X}$ such that $\Vert x- \hat{x}\Vert \leq \mu R$. Therefore $\Vert x- \pi_X(\hat{x})\Vert \leq (\delta_R+\mu)R$. If $\hat{x}\in S$ then the second condition of \Cref{def:adapt} is satisfied directly. Otherwise, there is $\hat{x_0}\in S$ such that $\Vert\hat{x_0} - \hat{x}\Vert \leq \lambda\widehat{\lfs}(\hat{x_0}) \leq \lambda\lfs(\hat{x_0})$. We have that
\[
\Vert x - \pi_X(\hat{x_0}) \Vert \leq \Vert x - \pi_X(\hat{x})\Vert + \Vert \pi_X(\hat{x})- \pi_X(\hat{x_0}) \Vert
\]
\[
\leq \Vert x - \pi_X(\hat{x})\Vert + \Vert \pi_X(\hat{x_0}) - \hat{x_0}\Vert + \Vert \hat{x_0} - \hat{x}\Vert + \Vert \hat{x} - \pi_X(\hat{x}) \Vert
\]
\[
\leq (\delta_R+\mu) R + \delta_R R + \lambda\lfs(\hat{x_0}) + \delta_R R
\]
\[
= R(\mu+3\delta_R) + \lambda(\lfs(\hat{x_0}) - \lfs(\pi_X(\hat{x_0}))) + \lambda\lfs(\pi_X(\hat{x_0}))
\]
\[
\leq \lfs(\pi_X(\hat{x_0}))(\mu + 3\delta_R + \lambda) + \lambda\Vert \hat{x_0} - \pi_X(\hat{x_0})\Vert
\]
\[
\leq \lfs(\pi_X(\hat{x_0}))(\mu + 3\delta_R + \lambda(1+\delta_R))\text{.}
\]
\end{proof} 

\begin{definition}
Let $X \subseteq\mr^\numvars$ be compact and the real part of a smooth and equidimensional algebraic variety, let $P$ be a finite subset of $\mr^\numvars$, and let $\widehat{\lfs}:\mr^\numvars\to\mr$ be the local feature size oracle for $X$ described in \Cref{sec:lfs_compute}. For $\kappa$ with $0 < \kappa$ and any $z\in\mr$, let $\widehat{\lfs}_\kappa = \kappa\widehat{\lfs}(z)$. Denote by $K_{P,\widehat{\lfs}}(\kappa)$ the union of balls 
\[
\cup_{z\in P} \overline{B}_z(\widehat{\lfs}_\kappa(z))
\]
and denote by $C_{P,\widehat{\lfs}}$ the functor $\RR\to\simp$ where $C_{P,\widehat{\lfs}}(\kappa)$ is the nerve of $\{\overline{B}_z(\widehat{\lfs}_\kappa(z))\}_{z\in P}$.
\end{definition}
\begin{remark} Let $0 \leq \alpha',\beta' < 1$ and let $\hat{X}$ be an adaptive-$(\alpha',\beta')$ sample of $X$ in $\mr^\numvars$. Since $\lfs$ is 1-Lipschitz we have for any $\hat{x}\in\hat{X}$ that $\lfs(\hat{x}) - \lfs(\pi_X(\hat{x})) \leq \Vert \hat{x} - \pi_X(\hat{x})\Vert \leq \alpha'\beta'\lfs(\pi_X(\hat{x}))$. This rearranges to $\lfs(\pi_X(\hat{x})) \leq \frac{1}{1-\alpha'\beta'}\lfs(\hat{x})$.
\end{remark}
\begin{proposition}
Let $\hat{X}$ be an adaptive-$(\alpha',\beta')$ sample of $X$ with $0\leq\alpha'\beta' < 1$ and let $\hat{R} = \min_{\hat{x}\in X} \hat{\lfs}(\hat{x})$. Set $M_{K} = \frac{E_{\lfs}}{R} + \alpha'\beta' + 1$ and $M_{\hat{K}} = \frac{E_{\lfs}}{\hat{R}(1-\alpha'\beta')} + \frac{1}{1-\alpha'\beta'}$. Then for any $\kappa > 0$, $K_{\hat{X},\lfs^\pi}(\kappa) \subseteq K_{\hat{X},\widehat{\lfs}}(M_{\hat{K}}\kappa)$ and $\hat{K}_{\hat{X},\widehat{\lfs}}(\kappa)\subseteq K_{\hat{X},\lfs^\pi}(M_K\kappa).$ \label{prop:estimated_adaptive_interleaving}
\end{proposition}
\begin{proof}
For the first inclusion, suppose $y\in\mr^\numvars$ has $\Vert y - \hat{x} \Vert \leq \kappa\lfs(\pi_X(\hat{x}))$ for some $\hat{x}\in\hat{X}$. Hence,
\[
\Vert y - \hat{x} \Vert \leq \frac{\kappa}{1-\alpha'\beta'}\lfs(\hat{x}) = \frac{\kappa}{1-\alpha'\beta'}(\lfs(\hat{x}) - \widehat{\lfs}(\hat{x}) + \widehat{\lfs}(\hat{x}))
\]
\[
\leq \frac{\kappa}{1-\alpha'\beta'}(E_{\lfs} + \widehat{\lfs}(\hat{x}))\leq \frac{\kappa\widehat{\lfs}(\hat{x})}{1-\alpha'\beta'}\left(\frac{E_{\lfs}}{\hat{R}} + 1\right)\text{.}
\]
For the second inclusion, suppose $y\in\mr^\numvars$ has $\Vert y - \hat{x} \Vert \leq \kappa\widehat{\lfs}(\hat{x})$ for some $\hat{x}\in\hat{X}$. Then 
\[
\Vert y - \hat{x} \Vert \leq \kappa(\widehat{\lfs}(\hat{x}) - \lfs(\hat{x}) + \lfs(\hat{x})) \leq \kappa(E_{\lfs} + \lfs(\hat{x}) - \lfs(\pi_X(\hat{x})) + \lfs(\pi_X(\hat{x}))) 
\]
\[
\leq \kappa(E_{\lfs} + \Vert \hat{x} - \pi_X(\hat{x})\Vert + \lfs(\pi_X(\hat{x}))) \leq \kappa(E_{\lfs}+\alpha'\beta'\lfs(\pi_X(\hat{x})) + \lfs(\pi_X(\hat{x})))
\]
\[
\leq \kappa\lfs(\pi_X(\hat{x}))\left(\frac{E_{\lfs}}{R}+\alpha'\beta'+1\right).
\]
\end{proof}
\begin{corollary}
With notation and assumptions as in \Cref{prop:estimated_adaptive_interleaving}, let 
$$M_K = \left(\frac{E_{\lfs}}{\hat{R}} + 1\right)\frac{1}{1-\alpha'\beta'}
\hbox{~~and~~} M_{\hat{K}} = \frac{E_{\lfs}}{R}+\alpha'\beta'+1.$$ Suppose that $0 < a < (M_K M_{\hat{K}})^2 a  < \frac{1}{3} - \alpha'\beta'$ and $$g(\alpha',\beta',a,(M_k M_{\hat{K}})^2 a) \leq h(\alpha',\beta',a,(M_k M_{\hat{K}})^2 a).$$ 
Then, the Betti number $\beta_\ell(X)$ is the rank of the map obtain from applying $H_\ell$ to the inclusion $C_{\hat{X},\widehat{\lfs}}(M_{K}a) \subseteq C_{\hat{X},\widehat{\lfs}}(M_{K}^2 M_{\hat{K}} a)$. \label{cor:cech_adapt_inference}
\end{corollary}
\begin{proof}
The former Proposition gives us the chain of inclusions
\[
{K}_{\hat{X},\lfs^\pi}(a) \subseteq K_{\hat{X},\widehat{\lfs}}(M_{\hat{K}} a) \subseteq {K}_{\hat{X},\lfs^\pi}(M_{\hat{K}}M_K a) \subseteq K_{\hat{X},\widehat{\lfs}}(M_{\hat{K}}^2 M_{K}a) \subseteq {K}_{\hat{X},\lfs^\pi}((M_{\hat{K}}M_K)^2 a)\text{.}
\]
By applying \Cref{thm:adapt_hom_inference} we have that all the spaces $K_{\hat{X},\lfs^\pi}$ deformation retract to $X$. Applying the nerve lemma we can replace $K$ with $C$. 
\end{proof}

\begin{proposition}
With notation and assumptions as in \Cref{cor:cech_adapt_inference} except with $4(M_KM_{\hat{K}})^2$ replacing $(M_KM_{\hat{K}})^2$, the Betti number $\beta_\ell(X)$ is the rank of the map obtained from applying $H_\ell$ to the inclusion $R_{\hat{X},\widehat{\lfs}}(M_K a) \subseteq R_{\hat{X},\widehat{\lfs}}(2 M_K^2M_{\hat{K}} a)$.
\end{proposition}
\begin{proof}
For any $\kappa$ we have the following inclusions. Use them to produce a 6-term inclusion chain similarly to the proof of \Cref{thm:hom_inference}:
\[C_{\hat{X},\widehat{\lfs}}(\kappa) \subseteq R_{\hat{X},\widehat{\lfs}}(\kappa) \subseteq C_{\hat{X},\widehat{\lfs}}(2\kappa)\text{.}
\]
\end{proof}

\end{document}